\documentclass{amsart}
\theoremstyle{plain}
\usepackage{enumitem}
\usepackage{nicefrac}
\usepackage{amsmath}
\usepackage{bbm}
\usepackage[multiple]{footmisc}
\usepackage{amsthm}
\usepackage{amssymb}
\usepackage{xfrac}
\usepackage{array}
\usepackage{mathtools}
\usepackage{stmaryrd}
\usepackage{float}
\usepackage{xcolor}
\usepackage{tikz-cd}

\newtheorem{prop}{Proposition}

\newtheorem{theorem}[prop]{Theorem}

\newtheorem{cor}[prop]{Corollary}
\newtheorem{lemma}[prop]{Lemma}

\numberwithin{prop}{section}
\numberwithin{equation}{section}

\newtheorem*{definition*}{Definition}
\newtheorem*{claim*}{Claim}
\newtheorem*{theorem*}{Theorem}
\newtheorem*{prop*}{Proposition}
\newtheorem*{example*}{Example}
\newtheorem*{remark*}{Remark}

\DeclareMathOperator{\arcsinh}{arcsinh}

\newcommand{\RR}{\ensuremath{\mathbb{R}}}

\newcommand{\NN}{\ensuremath{\mathbb{N}}}
\newcommand{\ZZ}{\ensuremath{\mathbb{Z}}}

\newcommand{\editA}[1]{ #1}

\title{Surface groups are  flexibly stable}

\author{Nir Lazarovich, Arie Levit and Yair Minsky}

\thanks{The third named author  was partially supported by NSF grant DMS-1610827.}

\begin{document}

\begin{abstract}
We show that surface groups are  flexibly stable in permutations. 
This is the first non-trivial example of a non-amenable flexibly stable group.
Our method is  purely geometric and relies on an analysis of branched covers of hyperbolic surfaces. Along the way we establish a  quantitative variant  of   the LERF property for surface groups which may be  of independent interest.
\end{abstract}

\maketitle

\section{Introduction}
\label{sec:intro}

Stability in group theory studies  when almost-actions of groups are close to honest actions. This recent notion has been investigated for various types of actions and groups. Of particular interest is the stability of actions by permutations on finite sets, as motivated by the outstanding open problem of   whether every group is sofic. For more on the relation to soficity and existing results we refer to the end of this introduction.


\subsection*{Flexible stability}
Let $G  $ be a finitely presented group with set of generators $\Sigma$. Let $F_\Sigma$ be a free group with basis $\Sigma$ and   $\pi_\Sigma : F_\Sigma \to G$ be the natural quotient map. Consider the space $\mathrm{Hom}(F_\Sigma, \mathrm{S}_N)$ of homomorphisms from $F_\Sigma$ into the finite symmetric group $\mathrm{S}_N$ for some   $N \in \mathbb{N}$.

Roughly speaking the group $G$ is flexibly stable if any $\rho\in\mathrm{Hom}(F_\Sigma,\mathrm{S}_N)$   that almost factors through $\pi_\Sigma$ is close to some $\rho' \in\mathrm{Hom}(F_\Sigma,\mathrm{S}_M)$ that actually factors through $\pi_\Sigma$ for some $M \in \NN$ potentially slightly larger than $N$.


To make the above notion precise equip the symmetric group $\mathrm{S}_N$  with  the bi-invariant \emph{normalized Hamming  metric}  $\mathrm{d}_N$ given by
$$ \mathrm{d}_N \left( \sigma_1, \sigma_2 \right) = \frac{1}{N}   | \{ i \in \{1,\ldots, N\}\: : \:  \sigma_1(i) \neq  \sigma_2(i) \} | $$
for any pair of elemets $\sigma_1, \sigma_2 \in \mathrm{S}_N$.
The  space $\mathrm{Hom}(F_\Sigma, \mathrm{S}_N)$  then admits the corresponding metric
$$ \mathrm{d}_N \left( \rho_1, \rho_2 \right) =  \sum_{\sigma\in\Sigma} \mathrm{d}_N \left(\rho_1(\sigma), \rho_2(\sigma) \right)  $$
for any pair of homomorphisms $\rho_1, \rho_2 \in \mathrm{Hom}(F_\Sigma, \mathrm{S}_N)$.  
Consider the map $$\mathcal{E}_{N}^{M} :  \mathrm{Hom}(F_\Sigma, \mathrm{S}_N) \to  \mathrm{Hom}(F_\Sigma, \mathrm{S}_M)$$ defined for every    $N,  M \in \mathbb{N}$ with $N \le M$   by extending a given permutation representation of $F_\Sigma$ to act trivially on the  extra $M - N$  points. 

\begin{definition*}
  The group $G$ is \emph{flexibly stable in permutations} if it admits a   finite presentation $G = \left<\Sigma|R\right>$ with the following property: for any $\varepsilon > 0$ there is some $\delta  = \delta(\varepsilon) > 0$ such that for any  $N \in \mathbb{N}$ and $\rho \in \mathrm{Hom}(F_\Sigma,\mathrm{S}_N)$  satisfying 
$$ \mathrm{d}_N(\rho(r), \mathrm{id}_N) <  \delta \quad \forall r \in R $$
there exists $M \in \mathbb{N}$ with $ N \le M \le (1+\varepsilon) N$ and  $\rho' \in \mathrm{Hom}(F_\Sigma, \mathrm{S}_{M})$ that factors through $\pi_\Sigma$ and satisfies 
$$ \mathrm{d}_{M} \left(\mathcal{E}_N^{M} (\rho), \rho' \right) < \varepsilon. $$ 
\end{definition*}

 We remark that if the above definition   holds with respect to one finite presentation of the group $G$ then indeed it holds for all such presentations  \cite{arzhantseva2015almost}.    
 \vspace{5pt}
 
  The following is our main result.




\begin{theorem}
\label{thm:main result}
Let $S$ be a closed orientable surface of genus $g \ge 2$. Then the fundamental group $\pi_1(S)$  is flexibly stable in permutations.
\end{theorem}

In fact   our  proof of Theorem \ref{thm:main result}   gives an explicit relation between $\varepsilon$ and $\delta$. Up to constants    $\varepsilon = \delta \ln(1/\delta)  $.  This means that $\delta = o(\varepsilon)$ and $\varepsilon^2 = o(\delta)$ in the limit as both $\varepsilon $ and $\delta$  go to $0$.
 
\subsection*{On our method}

Our approach to proving flexible stability is purely geometric. 

We use   covering space theory to reformulate the problem in terms of  certain  branched covers of $S$. In this language the goal is to convert a given branched cover of $S$ into an unramified one by performing an amount of changes controlled by the total branching degree. For more information we refer the reader to Theorem \ref{thm:surface groups are stable} and the discussion leading to it.

We dissect a  given branched cover $X$ of $S$ along a carefully constructed  embedded   graph $\Gamma$ with vertices at singular points of $X$ and with geodesic edges. In fact   $\Gamma$ is  taken to be a sub-graph of the Delaunay graph with respect to the singular points. It is chosen in such a way that every  resulting connected component of $X \setminus \Gamma$ has locally convex boundary and that the total length of all boundary curves is controlled by the total branching degree of $X$. See Theorem \ref{thm:existance of cut graphs}.  

The desired unramified cover of $S$ is  obtained by individually embedding every connected component of $X\setminus \Gamma$  into an unramified cover of $S$ with  controlled area. This relies on the following result,  which seems to be of independent interest. 


\begin{theorem}
\label{thm:variant of LERF}
Let $S$ be a closed hyperbolic surface.   Let $R$ be a   surface with boundary which is isometrically embedded in some cover of $S$. If the boundary $\partial R$ is locally convex  then  $R$ can be isometrically  embedded in a cover $Q$ of $S$  such that  the diagram
\begin{equation}
\label{eq:commute R}
  \begin{tikzcd}[column sep=tiny,row sep=small]
R \arrow[rr, hook] \arrow[rd]& &Q \arrow[ld]\\
&S&
\end{tikzcd}
\tag{$\mathrm{A}$}
\end{equation}
commutes and such that
$$\mathrm{Area}(Q \setminus R ) \le  b_S l(\partial R)$$
where $b_S > 0$ is a constant depending only on the surface $S$.
\end{theorem}

Surface groups are locally extended residually finite (LERF). This fact   was established by Scott in \cite{scott1978subgroups, scott1985correction}. It implies that any surface $R$ as above isometrically embeds in some \emph{finite} cover of $S$. Therefore Theorem \ref{thm:variant of LERF} can be regarded as a certain quantitative variant of the LERF property for surface groups. As such it is closely related  to and is  inspired by Patel's work \cite{patel2014theorem}. It is however crucial for our purposes  that the upper bound depends on $l(\partial R)$ rather then on $\mathrm{Area}(R)$.

A finitely presented group is   \emph{strictly stable in permutations} if it satisfies the    definition of flexible stability with $M$ being exactly equal to $N$. The methods we use towards proving Theorem \ref{thm:variant of LERF} will in general  increase the total  area. For this reason we are only able to  establish flexible rather than strict stability.

This note includes an appendix on Voronoi cells and Delaunay graphs constructed on "hyperbolic planes with singularities".

\subsection*{Soficity, stability and related works}
There has been a significant recent interest in the notion of sofic groups. Roughly speaking a group is sofic if it can be approximated   by almost-actions on finite sets, see e.g. \cite{weiss2000sofic, elek2006sofic} for formal definitions. Groups known to be sofic include amenable and residually finite ones. An outstanding open problem asks whether every group is sofic? 

A part of the interest in stability  stems from the observation made in \cite{glebsky2009almost} that a non  residually finite group which is stable  in permutations cannot be sofic. This observation extends to flexible stability.


Groups that are currently known to be strictly stable in permutations   include finite groups by Glebsky and Rivera \cite{glebsky2009almost}, finitely generated abelian groups by Arzhantseva and P{\u{a}}unescu \cite{arzhantseva2015almost} and polycyclic and Baumslag--Solitar $\mathrm{BS}(1,n)$ groups by Becker, Lubotzky and Thom \cite{becker2018stability}.


 Becker and Lubotzky \cite{becker2018group}   proved that a group $G$  with  Kazhdan's property (T) is not strictly stable by removing   a single point from an action of  $G$ on a finite set. This strategy has led    Becker and Lubotzky to introduce the flexible notion of   stability in permutations  and  ask whether some Kazhdan groups might still be stable in the flexible sense. The question of strict and flexible stability  for surface groups is discussed in \cite{becker2018group} as well.

The importance of the question concerning flexible stability for some Kazhdan groups is highlighted by the recent  work of Bowen and Burton \cite{bowen2019flexible}. They prove the existence of a non-sofic group, conditioned on the assumption that the group $\mathrm{PSL}_d(\ZZ)$ is flexibly stable for some $d \ge 5$.

The fundamental group of a closed orientable surface with constant non-negative curvature is either finite or abelian and as such   stable respectively by \cite{glebsky2009almost} and \cite{arzhantseva2015almost}.   Our Theorem \ref{thm:main result} completes the picture for all closed orientable surfaces, at least as far as flexible stability is concerned.

Free non-abelian groups are clearly stable in a void sense. On the other hand there exist small cancellation hyperbolic groups that are not even flexibly stable \cite{becker2018group}. Interestingly, while hyperbolic  surface groups are   possibly the simplest  example of  non-free hyperbolic groups, the question of stability for these groups is not trivial.
To the best of the authors' knowledge it is not known whether surface groups are stable in permutations in the strict sense. The problem of adapting  our present approach to deal with this question seems challenging. 

Finally we point out that it was recently shown by Becker and Mosheiff \cite{becker2018abelian} that for the free abelian group $ \mathbb{Z}^d$ the parameter $\delta = \delta(\varepsilon)$ goes to zero at least as fast as a polynomial of degree $d$ in  $\varepsilon$. So in some quantitative sense hyperbolic surface groups are "more stable" than free abelian ones.

\subsection*{Acknowledgments}

The authors would like to thank  Goulnara Arzhantseva, Nir Avni, Oren Becker, Tsachik Gelander and Alex Lubotzky for sharing  their inspiring ideas and for useful discussions. \editA{The authors would like to thank the anonymous referee for their useful corrections and suggestions.}




\section{$*$-covers and geometric stability}
\label{sec:*-covers and stability}

We develop a geometric framework for the study of flexible stability of surface groups  in terms of branched covers of surfaces.


\subsection*{$*$-covers of surfaces}
Let $S$ be a  closed surface. Endow $S$ with a metric of constant sectional curvature.
Recall that a finite \emph{branched cover} of $S$ is a continuous surjection $p : X \to S$ where every point $b \in S$ admits a neighborhood $b \in V_b \subset S$ with $p^{-1}(V_b) = U_1 \sqcup \cdots \sqcup U_{n_b}$ and such that  $p : U_i \twoheadrightarrow
 V_b$ is   topologically conjugate to the complex map $z \mapsto z^{d}$ of some  degree  $d \in \mathbb{N}$ depending on $i$. The degree   is equal to one unless $b$ belongs to a finite subset   of $ S$ called the \emph{branch set}. 


\begin{definition*}
A \emph{$*$-cover} of $S$ is a compact surface $X$ admitting a branched cover  $p : X \to S$ with  branch set   consisting of a single  branch point $* \in S$.
\end{definition*}

Let $p : X \to S$ be a $*$-cover. In general $X$   is not required to be connected. We pull back the metric from $S$ to every connected component of $X$ so that $p$ is a local isometry away from $p^{-1}(*)$. 

The total   angle $\alpha_x$ locally at any point  $x \in X$ is an integer multiple of $2\pi$. In fact $\alpha_x$ is equal to   $  2\pi d_x$ where  $d_x \in \NN$ is the \emph{degree} (or \emph{index}) of the point $x \in X$.


The \emph{singular set} of $X$ is $s(X) = \{x \in X : d_x > 1\}$. Clearly $s(X) \subset p^{-1}(*)$ and in particular $s(X)$ is discrete. The \emph{branching degree} of $X$  is   $ \beta(X) = \sum_{x \in s(X)} (d_x-1)$. 
We say that the $*$-cover $X$ is \emph{unramified} if $s(X) = \emptyset$, or equivalently if $\beta(X) = 0$. This happens if and only if $p$ is a covering map. We emphasize that covers are not required  to be connected.

The \emph{degree} $|X|$ of the $*$-cover $p: X \to S$ is equal to $|p^{-1}(x)|$ for any $x \in S \setminus \{*\}$. 

\subsection*{Geometric stability} The geometric notion of stability is defined in terms of certain  graphs embedded into $*$-covers.
Recall that a \emph{graph} is a simplicial $1$-complex. 
We let $\Gamma^{(0)}$ denote the vertex set of the graph $\Gamma$.

\begin{definition*}
A \emph{$*$-graph} on $X$ is  an embedded graph $\Gamma$   such that  $ \Gamma^{(0)} \subset p^{-1}(*)   $ and  the edges of $\Gamma$ are geodesic arcs.  
\end{definition*}

The fact that $\Gamma$ is embedded means that   two edges of $\Gamma$ may intersect only at the vertex set $\Gamma^{(0)}$.  If $\Gamma$ is a $*$-graph on $X$ then any closed curve $\gamma$ contained in $\Gamma$ is a piece-wise geodesic closed curve in $X$.  
Given a $*$-graph $ \Gamma $ let $l(\Gamma) $ denote the total length of all of its edges.

\begin{definition*}
The  $*$-cover $p : X \to S$ is  \emph{$\varepsilon$-reparable} if $X $   admits a   $*$-graph $\Gamma$  with  
$  l(\Gamma )  \le \varepsilon  \mathrm{Area}(X ) $
such that the complement  $X \setminus \Gamma$  isometrically embeds into a cover  $C$ of $S$ such that the diagram
\begin{equation}
\label{eq:embeds diagram}
  \begin{tikzcd}[column sep=tiny,row sep=small]
X \setminus \Gamma \arrow[rr, hook] \arrow[rd, "p"]& &C \arrow[ld]\\
&S&
\end{tikzcd}
\tag{$\mathrm{B}$}
\end{equation}
commutes and such that  
$$ \quad \mathrm{Area}(X) \le \mathrm{Area}(C) \le (1+\varepsilon)\mathrm{Area}(X).$$
\end{definition*}

We emphasize that the complement $X \setminus \Gamma$ as well as the cover $C$ are in general allowed to be   disconnected.

\begin{definition*}
The surface $S$ is  \emph{flexibly geometrically stable} 
  if for every $\varepsilon > 0$ there is some $\delta = \delta(\varepsilon) > 0$ such that any  $*$-cover $X$  with $\beta(X) < \delta \mathrm{Area}(X)$ is 
    $\varepsilon$-reparable.
\end{definition*}
Our main result can be reformulated in the language of $*$-covers.

\begin{theorem}
\label{thm:surface groups are stable}
Every closed  orientable hyperbolic surface  is flexibly geometrically stable.
\end{theorem}
In the remaining part of this section we show that geometric stability implies algebraic stability in permutations, see Proposition \ref{prop:geometric stability implies algebraic stability} below.



\subsection*{Permutation representations and covering theory}

Consider the punctured surface $S \setminus \{*\}$. Fix a base point $x_0 \in S \setminus \{*\}$ and let $F = \pi_1(S\setminus\{*\}, x_0)$ so that $F$ is a free group. 
Note that $X \setminus p^{-1}(*) \to S \setminus \{*\}$ is a covering map of degree $|X|$. Enumerate  the fiber  $p^{-1}(x_0)$ by fixing an arbitrary identification  with the finite set $\{1,\ldots,|X|\}$. Covering space theory \cite[p. 68]{hatcher2002algebraic} shows how to associate  a natural permutation representation $\rho_X : F \to \mathrm{S }_{|X|}$ to the covering $X \setminus p^{-1}(*) \to S \setminus \{*\}$. 


In what follows it is  convenient  to fix a specific presentation for the fundamental group of $S$. Since the notion of flexible stability is known to be independent of the chosen presentation we may do so without any loss of generality.

Let $\mathcal{P} $ be a   compact fundamental domain for the action of the fundamental group $\pi_1(S, x_0)$ on the universal cover $\widetilde{S}$. Assume that  $\mathcal{P}$ is a     polygon with finitely many geodesic sides and that the vertices of $\mathcal{P}$ are all lifts of the     point $ * \in S$. Moreover assume that the lift of the point $x_0$ lies in the interior of $\mathcal{P}$.

Let $\Sigma$ be the finite generating set of $\pi_1(S,x_0)$ consisting of a single element $\sigma$ from any pair $\{\sigma, \sigma^{-1}\}   \subset \pi_1(S,x_0)$ such that  $\sigma \mathcal{P}\cap \mathcal{P}$ is a geodesic side of $\mathcal{P}$.  We may identify $F$ with the free group $F_\Sigma$ on the generators $\Sigma$. There is a natural quotient homomorphism $\pi_\Sigma : F_\Sigma \to \pi_1(S,x_0)$. The surface group $\pi_1(S,x_0)$ admits a presentation with generating set $\Sigma$ and a single relation $r \in F_\Sigma$. 

In particular we may let the fundamental domain $\mathcal{P}$   be a $4g$-sided polygon so that $\Sigma = \{a_1,b_1,\ldots,a_g,b_g\}$ and $r = \prod_{i=1}^g \left[a_i,b_i\right]$.

\begin{prop}
\label{prop:construction of *-cover from a homomorphism}
Given any $\rho \in \mathrm{Hom}(F_\Sigma, \mathrm{S}_N)$ with $N \in \mathbb{N}$ there exists a $*$-cover $X_\rho \to S$ with $|X_\rho| = N$ and $\rho_{X_\rho} = \rho$. Moreover 
$ \frac{\beta(X_\rho)}{N} \le   \mathrm{d}_N(\rho(r), \mathrm{id}_N)$ 
where $r$ is the defining relation of $\pi_1(S,x_0)$ as above.
\end{prop}
Note that $\mathrm{d}_N(\rho(r), \mathrm{id}_N)$ is equal to $1 - \frac{|\mathrm{Fix}(\rho(r))|}{N}$ where $\mathrm{Fix}(\rho(r))$ is the set of fixed points of the permutation $\rho(r) \in \mathrm{S}_N$.
\begin{proof}[Proof of Proposition \ref{prop:construction of *-cover from a homomorphism}]
Let $X'_\rho$ be the punctured surface covering $S \setminus \{*\}$  that  corresponds to the permutation representation $\rho : F_\Sigma \to \mathrm{S}_N$.   Let $p : X_\rho \to S$ be the $*$-cover   obtained by   completing   $X'_\rho$ at the punctures. It is clear that $|X_\rho| = N$ and $\rho_{X_\rho}= \rho$. 

Consider the cycle decomposition $ o_1 \cdots o_m$ of the permutation $\rho(r)$ in its action on the  set $p^{-1}(x_0) \cong \{1,\ldots,N\}$. Let $l_i$ denote the length of the cycle $o_i$. We claim that without loss of generality $p^{-1}(*) = \{y_1,\ldots,y_m\}$ and $d_{y_i} = l_i$. This would imply the required upper bound on $\beta(X_\rho)$ as 
$$ \frac{\beta(X_\rho)}{N} = \frac{1}{N} \sum_{y\in p^{-1}(*)} (d_y-1) = \frac{1}{N} \sum_{i=1}^m (l_i -1) \le \mathrm{d}_N(\rho(r), \mathrm{id}_N).$$

Let $\gamma$ be a simple closed curve in $S \setminus \{*\}$ based at $x_0$ and representing the element $r$ of $F_\Sigma$. The preimage of the curve $\gamma$ in $X'_\rho$ is a disjoint union of simple closed curves.  Every such curve corresponds to an orbit $o_i$ of $\rho(r)$ in its action on $p^{-1}(x_0)$ and bounds a disc in $X_\rho$ that contains a single point $y_i$ from $p^{-1}(*)$ in its interior. Moreover the degree $d_{y_i}$  is equal to the size $l_i$ of the orbit $o_i$. The  claim follows.
\end{proof}

\subsection*{Geometric stability implies algebraic} We show  that geometric stability in the sense of Theorem \ref{thm:surface groups are stable} implies our main result Theorem \ref{thm:main result}. We continue using the presentation $\pi_1(S,x_0) \cong \left<\Sigma|r\right>$ constructed above and in particular the   polygon $\mathcal{P}$.

\begin{prop}
	\label{prop:distance between permutations separated by graph}
	Let $\Gamma$ be a  $*$-graph on $X$ with $N = |X|$. Assume that the complement $X \setminus \Gamma$ isometrically embeds into   some cover $C$ of $S$ of degree $M = |C|$ such that Diagram (\ref{eq:embeds diagram}) commutes.   Regard the two permutation representations $\rho_X$ and $\rho_C$  as mapping respectively into the   symmetric groups $\mathrm{S}_N$ and $\mathrm{S}_M$. 
 Then
$$ \mathrm{d}_{M}(\mathcal{E}_{N}^{M} (\rho_X), \rho_C) \le a_S \left(  \frac{l(\Gamma)+ \left(\mathrm{Area}(C) - \mathrm{Area}(X)\right)}{\mathrm{Area}(C)}\right)$$
where   $a_S > 0$ is a constant depending only on the surface $S$.
\end{prop}

Recall that the map $\mathcal{E}_{N}^{M} :  \mathrm{Hom}(F_\Sigma, \mathrm{S}_N) \to  \mathrm{Hom}(F_\Sigma, \mathrm{S}_M)$ is defined    by  extending a permutation representation to act trivially on the extra $M-N$ points.
\begin{proof}[Proof of Proposition \ref{prop:distance between permutations separated by graph}]
The map $p : X \to S$ induces a tessellation $\mathcal{T} = \mathcal{T}(X, \mathcal{P})$ of the $*$-cover $X$   by $|X|$-many isometric copies of the polygon $\mathcal{P}$.  
\editA{We claim that every edge $\alpha$ of the graph $\Gamma$ satisfies
\begin{equation*}
|\{ \mathcal{D} \in \mathcal{T} \: : \:  \mathcal{D} 
 \cap \alpha  \not\subseteq p^{-1}(*) \} | \le a'_{S} l(\alpha).
 \end{equation*}
where $a'_{S} > 0$ is a constant depending only on the surface $S$,  the choice of the fundamental domain $\mathcal{P}$ and the particular generating set $\Sigma$.

To prove the claim fix an edge $\alpha$ of the graph $\Gamma$.  This edge is a geodesic segment in $X$ such that $\alpha \cap p^{-1}(*) = \partial \alpha$.
We want to bound the number of polygons $\mathcal D$ of the tessellation $\mathcal{T}$ that $\alpha$ meets in its interior. 

If $\alpha$ has a subsegment that coincides with a subsegment of an edge of the tessellation $\mathcal T$, then $\alpha$ equals that edge (since both of them are geodesics which end in $p^{-1}(*)$). In this case  $\alpha$ meets the two polygons incident to this edge, and its length is at least the length $\ell$ of the shortest edge of $\mathcal P$. Therefore, $\alpha$  satisfies the desired inequality if $a'_S \ge 2/\ell$.

Otherwise, $\alpha$ meets the edges of the tessellation $\mathcal T$ transversely. Consider how $\alpha$ is cut by the boundary edges of the tessellation $\mathcal T$ into geodesic segments $\alpha_1,\dots,\alpha_k$. The number of polygons that $\alpha$ meets is bounded by $k$. 

Note that, since $\alpha$ does not pass through the singular points of $X$, $p\circ \alpha$ is a geodesic in $S$ (which is not necessarily embedded).  Each of $p\circ \alpha_i$ is an embedded geodesic in $S$ whose boundary is on the edges of $\mathcal Q$, where $\mathcal Q$ is the polygon in $S$ to which $\mathcal P$ projects under the universal covering map.

Fix $\varepsilon>0$ smaller than half the injectivity radius of $S$.
There exists $\delta>0$ (depending on $\mathcal{P}$ and $\varepsilon$) such that each $p\circ \alpha_i$ either has length at least $\delta>0$ or is contained in the $\varepsilon$-neighborhood of $*$.
This shows in particular that $l(\alpha)\ge\delta$.
Moreover, since the edges of $\mathcal Q$ are geodesic, there exists $m$ such that any geodesic segment in $S$ of length $2\varepsilon$ meets the edges of $\mathcal Q$ at most $m$ times. 
It follows that any $m+1$ consecutive segments in the sequence $p\circ \alpha_1,\dots,p\circ \alpha_k$ must leave the $\varepsilon$-neighborhood of $*$, and therefore  must have length at least $\delta$ altogether.
This shows that $\delta \frac{k}{2m+2} \le \max\{\delta, \delta \lfloor\frac{k}{m+1}\rfloor \}\le l(\alpha)$. So the edge $\alpha$  satisfies the desired inequality if $a'_S \ge \frac{2m+2}{\delta}$.
Choosing $a'_S = \max\{\frac{2}{\ell},\frac{2m+2}{\delta}\}$ completes the proof of the claim.
}

Taking into account the claim and summing  over all edges of the graph $\Gamma$  gives the estimate
$$ |\{ \mathcal{D} \in \mathcal{T} \: : \:  \mathcal{D} 
 \cap \Gamma  \not\subseteq p^{-1}(*) \} | \le a'_{S} l(\Gamma). $$
 
 Consider a point $x \in p^{-1}(x_0)$ and let $\mathcal{D} \in \mathcal{T}$ be the polygon containing $x$ in its interior. The fiber $p^{-1}(x_0)$ is clearly a subset of the fiber $q^{-1}(x_0)$ so that $x \in q^{-1}(x_0)$. Given a particular generator $\sigma \in \Sigma$ let $\mathcal{D}_\sigma \in \mathcal{T}$ be the polygon sharing with $\mathcal{D}$ the geodesic side corresponding to $\sigma$.
Then 
   $$\mathcal{E}_{N}^{M}(\rho_X)(\sigma)(x) = \rho_C(\sigma)(x)$$
    provided that the $*$-graph $\Gamma$ does not intersect $(\mathcal{D} \cup \mathcal{D}_\sigma)\setminus p^{-1}(*)$.
The Hamming metric between  $\mathcal{E}_{N}^{M} (\rho_X)$ and $\rho_C$ is therefore bounded above by
 $$ \mathrm{d}_{M}(\mathcal{E}_{N}^{M}(\rho_X),\rho_C) \le \frac{2|\Sigma| }{M}\left( a'_{S} l(\Gamma) +  |q^{-1}(x_0) \setminus p^{-1}(x_0)| \right) $$
 and the proof follows for an appropriate choice of the constant $a_S > 0$.
\end{proof}

\begin{prop}
\label{prop:geometric stability implies algebraic stability}
Let $S$ be a closed hyperbolic surface. If $S$ is  flexibly geometrically stable then its fundamental group $\pi_1(S, x_0)$ is flexibly  stable in permutations.
\end{prop}
\begin{proof}
The algebraic property of   flexible stability in permutations does not depend on the choice of a particular presentation. This   is proved in \cite{arzhantseva2015almost}  for the strict notion of stability  using on the fact that any two finite presentations for $G$ are related by a finite sequence of Tietze transformations \cite[II.2.1]{lyndon2015combinatorial}. Exactly the same argument goes through  in the flexible case  as well.
It will be convenient to verify the definition with respect to the presentation $\pi_1(S,x_0) \cong \left<\Sigma|r\right>$.


Let $\varepsilon > 0$ be given. Denote $\varepsilon' = \min\{ \varepsilon, \frac{\varepsilon}{2a_S}\}$. According to our assumption there is some $\delta = \delta(\varepsilon')$ such that any $*$-cover $X$ with $\beta(X) < \delta \mathrm{Area}(X)$ is $\varepsilon'$-reparable. We claim that the definition of flexible stability in permutations for the group $\pi_1(S,x_0)$ is satisfied with respect to the given $\varepsilon > 0$ and this particular $\delta$.

Consider some $N \in \mathbb{N}$ and  $\rho \in \mathrm{Hom}(F_\Sigma, \mathrm{S}_N)$ with $\mathrm{d}_N(\rho(r), \mathrm{id}_N) < \delta$. Making use of Proposition \ref{prop:construction of *-cover from a homomorphism} we find a $*$-cover $X_\rho \to S$ with $\beta(X_\rho) < \delta N$ and $\rho_{X_\rho}= \rho$. Therefore $X_\rho$ is $\varepsilon'$-reparable which means it admits an embedded  $*$-graph $\Gamma$ with  
$  l(\Gamma )  \le \varepsilon'  \mathrm{Area}(X_\rho) $
and such that the complement  $X_\rho \setminus \Gamma$ isometrically  embeds into some  cover $C$ of $S$ making Diagram (\ref{eq:embeds diagram}) commute and satisfying  
$$\mathrm{Area}(X_\rho) \le \mathrm{Area}(C) \le (1+\varepsilon')\mathrm{Area}(X_\rho).$$ 

Consider the permutation representation $\rho' = \rho_C$ so that $\rho' \in \mathrm{Hom}(F_\Sigma, \mathrm{S}_{M})$ where $M = |C|$. Since $C$ is unramified the homomorphism $\rho'$ factors through $\pi_\Sigma : F_\Sigma \to \pi_1(S, x_0)$. Roughly speaking, the two permutations representations $\rho$ and $\rho'$ agree on points lying on the subsurface $X \setminus \Gamma$ of $C$ away from the $*$-graph $\Gamma$. To be precise it follows from Proposition \ref{prop:distance between permutations separated by graph} that 
$$\mathrm{d}_{M}(\mathcal{E}_N^M (\rho), \rho') \le a_S \left( \frac{l(\Gamma) + \left( \mathrm{Area}(C) - \mathrm{Area}(X_\rho)\right)}{\mathrm{Area}(C)}\right) \le 2 a_S \varepsilon' \le \varepsilon.$$
Finally the degree $M$  satisfies
$$N \le M \le (1+\varepsilon')N \le (1+\varepsilon)N$$ 
as required.
\end{proof}

The remainder of this work is   dedicated to  establishing Theorem \ref{thm:surface groups are stable}.

\section{Hyperbolic planes with singularities}
\label{sec:hyperbolic planes with singularity}

Let $S$ be a compact hyperbolic surface and consider a fixed  $*$-cover $p : X \to S$. 
We discuss the geometry of the universal cover of $X$.  We then discuss the Voronoi tessellation and its dual the  Delaunay graph. This graph   will be used in Section \ref{sec:cut graphs}   to construct the $*$-graph as required in the definition of flexible geometric stability.

\subsection*{The geometry of the universal  cover of $X$} Let $q : \widetilde{X} \to X$ denote the universal cover of $X$ equipped with the pullback length metric $d_{\widetilde{X}}$. Topologically speaking $\widetilde{X}$ is homeomorphic to the plane. 

The singular set of $\widetilde{X}$ is  $s(\widetilde{X}) = q^{-1}(s(X))$. The   space  $\widetilde{X}$   is locally isometric to the hyperbolic plane $\mathbb{H}$ away from its singular set.  The group $\pi_1(X)$ acts freely on $\widetilde{X}$ admitting the surface $X$ as a quotient. The singular set  $s(\widetilde{X})$ is discrete in $\widetilde{X}$. In fact $s(\widetilde{X})$   is   co-bounded provided $s(X)$  is not empty. 

Recall that the Cartan--Hadamard theorem admits a  generalization due to Gromov   to complete geodesic metric spaces. See  \cite[Theorem II.4.1]{bridson2013metric} for reference.  This result  implies that   $\widetilde{X}$ is a  $\mathrm{CAT}(-1)$-space. 
In particular $\widetilde{X}$ is \emph{uniquely geodesic} 
and every  \emph{local geodesic in $\widetilde{X}$ is a geodesic}. 

Let $\gamma$ be a continuous path in $\widetilde{X}$. The path  $\gamma$   is a \emph{local geodesic} provided it is a local geodesic in the  sense of hyperbolic geometry away from the singular set $s(\widetilde{X})$ and its angle $\theta_x$ at every singular point $x \in s(\widetilde{X})$ along $\gamma$ satisfies $\pi \le \theta_x \le \alpha_x - \pi$ where $\alpha_x$ is the local angle at $x$.
This local characterization  implies that   $\widetilde{X}$ is \emph{geodesically complete} in the sense that  any geodesic segment can   be extended to a bi-infinite geodesic line. The geodesic extension  need not be unique since a geodesic segment terminating at a singular point admits many extensions.

\subsection*{Convex subsets of the universal cover $\widetilde{X}$}
Recall the following useful consequence of the Cartan--Hadamard theorem \cite[II.4.13 and II.4.14]{bridson2013metric}.

\begin{lemma}
	\label{lem:lifting local isometry}
	Let $N_1$ and $N_2$ be connected complete non-positively curved metric spaces and  $f : N_1 \to N_2$ be a local isometry.  Then $f_* : \pi_1(N_1) \to \pi_1(N_2)$ is injective and any lift $F : \widetilde{N}_1 \to \widetilde{N}_2$ of the map $f$ is an isometric embedding.
\end{lemma}

Here is an example of a straightforward application of Lemma \ref{lem:lifting local isometry}.

\begin{cor}
\label{cor:convex subset embeds into H}
Let $C \subset \widetilde{X}$ be a convex subset such that $\mathring{C} \cap s(\widetilde{X}) = \emptyset$. Then $C$ is isometric to a contractible subset of the hyperbolic plane $\mathbb{H}$.
\end{cor}
\begin{proof}
Since $C$ is a convex subset of $\widetilde{X}$ it is a $\mathrm{CAT}(-1)$-space in its own right. The corollary follows by applying Lemma \ref{lem:lifting local isometry} with $N_1 = C, N_2 = S$ and $f = p \circ q$. The lift $F$ gives the required isometric embedding into $\mathbb{H}$.
\end{proof}

Let $\gamma$ be a bi-infinite geodesic path in $\widetilde{X}$. A \emph{half-space} in $\widetilde{X}$ is  the closure of a connected component of $ \widetilde{X} \setminus \gamma$. Note that a half-space is convex.

\subsection*{Voronoi cells and the Delaunay graph}

Assume that the singular set $s(X)$ is non-empty. We   define the Voronoi cells and the Delaunay graph with respect to the set of points $s(\widetilde{X})$. These are natural generalizations of the parallel notions in the classical Euclidean and hyperbolic cases.

\begin{definition*}
\label{def:Voronoi cell}
The \emph{Voronoi cell} $A_v$ at the singular point $v \in s(\widetilde{X})$ is 
$$ A_v = \{x \in \widetilde{X}\: : \: d_{\widetilde{X}}(x,v) \le d_{\widetilde{X}}(x,u)\quad \forall u \in s(\widetilde{X})   \}.$$
\end{definition*}

The family of Voronoi cells is equivariant in the sense that $g A_v = A_{gv}$ for every $g \in \pi_1(X)$ and $v \in s(\widetilde{X})$. 

We remark that  it is not a priori clear that the Voronoi cells form a tessellation. This turns out to be true   and will be   established as a consequence of Proposition \ref{prop:two Voronoi cells are disjoint or share a common side} below. The  difficulty has to do with the fact that for two singular points $v,u \in s(\widetilde{X})$ the intersection of the two sets   $\{x \in \widetilde{X} :  d_{\widetilde{X}}(x,v) \le d_{\widetilde{X}}(x,u)\}$ and  $ \{x \in \widetilde{X}:  d_{\widetilde{X}}(x,v) \ge d_{\widetilde{X}}(x,u) \}$   might  have a non-empty interior in general.

\begin{definition*}
\label{def:Delaunay graph}
The vertex set of the   \emph{Delaunay graph}  $D$ is the singular set $s(\widetilde{X})$. Two vertices $v_1$ and $v_2$ of the Delaunay graph $ D$ span an edge  whenever there is a closed metric ball $B \subset \widetilde{X}$ with $\mathring{B} \cap s(\widetilde{X}) = \emptyset$ and $ B \cap s(\widetilde{X}) = \{v_1, v_2\}$. 
\end{definition*}


We summarize a  few  basic properties of   Voronoi cells and their relationship with    the Delaunay graph.   These are quite elementary in the classical  Euclidean or hyperbolic cases.   Extra caution is required     in the presence of singular points --- some of the following statements are false unless the Voronoi cells  are taken with respect to a set of points containing all   singularities, which is always the case here.

\medskip

\textit{The   proofs of  Propositions \ref{prop:properties of Voronoi and Delaunay}, \ref{prop:two Voronoi cells are disjoint or share a common side}, \ref{prop:Delaunay graph is embedded} and \ref{prop:the angle between two edges of Delaunay graph is at most pi} are postponed to Appendix  \ref{sec:properties of Voronoi and Delaunay}  dedicated to this topic. The proofs  are similar to the classical case by repeatedly  relying on Corollary \ref{cor:convex subset embeds into H} to embed the relevant local picture into the hyperbolic plane.}

\medskip

\begin{prop}
\label{prop:properties of Voronoi and Delaunay}
Let  $v \in s(\widetilde{X})$ be a vertex of the Delaunay graph $D$. Consider  the Voronoi cell $A_v$. Then
\begin{enumerate}
\item \label{it:Voronoi cell is homeomorphic to a disc}
$A_v$ is  homeomorphic to a closed disc and   $A_v \cap s(\widetilde{X}) = \{v\}$, 
\item \label{it:Voronoi has pw geodesic boundary}
$A_v$ is   convex and the boundary $\partial A_v$ is   piecewise geodesic, and
\item \label{it:Voronoi defined in terms of neighbors}
$A_v$ is equal to the intersection of the sets 
  $\{x \in \widetilde{X} :  d_{\widetilde{X}}(x,v) \le d_{\widetilde{X}}(x,u) \}$ where $u \in s(\widetilde{X})$ ranges over the vertices adjacent to  $v$ in the graph $D$.
\end{enumerate}
\end{prop}

The following proposition describes the   intersection of two  Voronoi cells.

\begin{prop}
\label{prop:two Voronoi cells are disjoint or share a common side}
Let $v,u \in s(\widetilde{X})$ be a pair of distinct vertices of the Delaunay graph. Consider the two Voronoi cells $A_v$ and $A_u$.
\begin{itemize}
 \item  If  $A_v \cap A_u \neq \emptyset$ then $A_v \cap A_u$ is either  a single point or a common geodesic side.
\item
   $A_v$ and $A_u$ have a common geodesic side if and only if    $v$ and $u$ span an edge  in the Delaunay  graph.
\end{itemize}
\end{prop}

As a consequence of Proposition \ref{prop:two Voronoi cells are disjoint or share a common side} we deduce that the family of   Voronoi cells $A_v$ for $v \in s(\widetilde{X})$  forms a tessellation of $\widetilde{X}$ called the \emph{Voronoi tessellation}. Moreover the Delaunay graph $D$ is   dual to this  tessellation. These facts are   well-known    in the classical Euclidean and hyperbolic situations.

\begin{cor}
\label{cor:Voronoi embeds}
Let  $v \in s(\widetilde{X})$ be a vertex of the Delaunay graph. Then the interior of the Voronoi cell $A_v$    embeds into $X$ via the restriction to $A_v$ of the covering map $q : \widetilde{X} \to X$.
\end{cor}
\begin{proof}
It suffices to observe that $g\mathring{A}_v \cap \mathring{A}_v = \mathring{A}_{gv} \cap \mathring{A}_v = \emptyset$ for every $g \in \pi_1(X)$ with $g \neq \mathrm{id}$.
\end{proof}

\subsection*{The geometric realization of the Delaunay graph}
From now on  we regard the Delaunay graph $D$  as being  \emph{geometrically realized} in   $\widetilde{X}$. More precisely,  we identify the vertices of $D$ with the singular set $s(\widetilde{X})$ and realize every edge of $D$ by the corresponding geodesic arc in $\widetilde{X}$. 

\begin{prop}
\label{prop:Delaunay graph is embedded}
The Delaunay graph $D$ is embedded in $\widetilde{X}$. The projection $q(D)$ is embedded in the surface $X$.
\end{prop}

The above statement means that     geodesic arcs realizing  two distinct edges of   $D$ may  intersect only at  a vertex incident to both. In particular $D$ is planar. Likewise two distinct edges of $q(D)$ may intersect only at a common singular point of $s(X)$.

\begin{prop}
\label{prop:the angle between two edges of Delaunay graph is at most pi}
Any connected component of the complement of the Delaunay graph in $\widetilde{X}$ has locally convex boundary.
\end{prop}

\subsection*{An estimate for   the area of Voronoi cells}


Recall that a subset $L$ of a metric space $M$ is   \emph{$r$-separated} if $d_M(x_1,x_2) \ge r$ for any two distinct points $x_1,x_2 \in L$.  Moreover recall   that we defined a half-space   in  $\widetilde{X}$ to be the closure of a connected component of the complement of some bi-infinite geodesic line in $\widetilde{X}$. 

\begin{lemma}
\label{lem:hyperbolic geometry lemma}
Let $  A_v$ be the Voronoi cell associated to the singular point $v \in s(\widetilde{X}) $.  Let $\mathfrak{H}$ be a half-space in $\widetilde{X}$ with $v \in \partial \mathfrak{H}$. Assume that $s(\widetilde{X})$ is $r$-separated and that $d_{\widetilde{X}}(v,u) \ge R$ for every singular point $u \in \mathfrak{H} \cap s(\widetilde{X})$ distinct from $v$.  Then
$$ \mathrm{Area}( \mathfrak{H} \cap A_v) \ge 2\varphi  \sinh^2(R/4) $$
where $\varphi > 0$ is a constant depending only on the distance $r$.
\end{lemma}

The expression appearing on the right-hand side of the above estimate  is  the   area of the hyperbolic sector with central angle $\varphi$ and of radius $R/2$. 
Throughout  the following proof and given a point $p \in \widetilde{X}$ it is convenient  to introduce the notation 
$$ \mathfrak{F}(p) = \{ x\in \widetilde{X}  : d_{\widetilde{X}}(x,v) \le d_{\widetilde{X}}(x,p) \}. $$

\begin{proof}[Proof of Lemma \ref{lem:hyperbolic geometry lemma}]

Let us first determine the angle  $\varphi > 0$   as follows. Let   $l_1$ and $l_2$ be a pair of geodesic lines in the hyperbolic plane $\mathbb{H}$ such that $d_{\mathbb{H}}(l_1, l_2) = r$.   Let $m$ be the mid-point of the geodesic arc perpendicular  to both $l_1$ and $l_2$. Then $\varphi$ is   the  angle  between the two geodesic rays emanating  from the point $m$ towards   the ideal points $l_1(\infty)$ and   $l_2(\infty)$.

Consider the bi-infinite geodesic $\rho : \mathbb{R} \to \widetilde{X}$ parametrized by arc length and such that $\rho(\mathbb{R}) = \partial \mathfrak{H}$.  Assume that $v = \rho(0)$ and denote $v_t = \rho(t)$ for all $t \in \RR$.

Let $\mathcal{W}_r^R$ be the subset of $\widetilde{X}$ given  by
$$\mathcal{W}_r^R = \mathfrak{H} \cap \mathfrak{F}(v_{-r}) \cap \mathfrak{F}(v_r) \cap B_{\widetilde{X}}(v,R). $$
As a   consequence of    Corollary \ref{cor:convex subset embeds into H}  the two convex sets  $B_{\widetilde{X}}(v,r)$ and $\mathfrak{H} \cap B_{\widetilde{X}}(v,R)$ are isometric to a hyperbolic ball of radius $r$ and a hyperbolic sector of radius $R$ and angle $\pi$, respectively. Let $Q_{\varphi, R/2}$ be the hyperbolic sector   of angle $\varphi$ and radius $R/2$ which is based at the vertex $v$ and   contained in $\mathcal{W}_r^R$.
The lemma will be established   by showing that $\mathcal{W}_r^R$ and therefore $Q_{\varphi, R/2}$ is contained in $\mathfrak{H} \cap A_v$.



Let $u_1, \ldots, u_n$  with $n \in \mathbb{N}$ be the vertices of the Delaunay graph $D$ adjacent to the vertex $v$.  For every $i \in \{1,\ldots,n\}$ let $w_i$  be the point along the geodesic arc from $v$ to $u_i$  such that $d_{\widetilde{X}}(v,w_i) = r$.
 The convexity of the metric \cite[II.2.2]{bridson2013metric} implies that $\mathfrak{F}(w_i) \subset \mathfrak{F}(u_i)$ and likewise $\mathfrak{F}(v_{\pm r}) \subset \mathfrak{F}(v_{\pm 2r})$.





\begin{figure}[ht!] 
\begin{center}
\def\svgwidth{0.4\textwidth}
\begingroup%
\makeatletter%
\providecommand\color[2][]{%
	\errmessage{(Inkscape) Color is used for the text in Inkscape, but the package 'color.sty' is not loaded}%
	\renewcommand\color[2][]{}%
}%
\providecommand\transparent[1]{%
	\errmessage{(Inkscape) Transparency is used (non-zero) for the text in Inkscape, but the package 'transparent.sty' is not loaded}%
	\renewcommand\transparent[1]{}%
}%
\providecommand\rotatebox[2]{#2}%
\ifx\svgwidth\undefined%
\setlength{\unitlength}{221.37226255bp}%
\ifx\svgscale\undefined%
\relax%
\else%
\setlength{\unitlength}{\unitlength * \real{\svgscale}}%
\fi%
\else%
\setlength{\unitlength}{\svgwidth}%
\fi%
\global\let\svgwidth\undefined%
\global\let\svgscale\undefined%
\makeatother%
\begin{picture}(1,0.82519136)%
\put(0,0){\includegraphics[width=\unitlength,page=1]{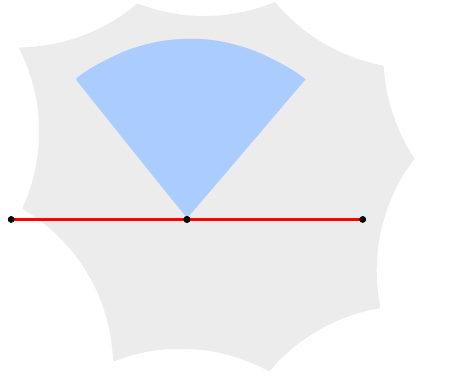}}%
\put(0.39666585,0.29459682){\color[rgb]{0,0,0}\makebox(0,0)[lb]{\smash{$v$}}}%
\put(0.76628011,0.29459682){\color[rgb]{0,0,0}\makebox(0,0)[lb]{\smash{$v_r$}}}%
\put(-0.00374969,0.29459682){\color[rgb]{0,0,0}\makebox(0,0)[lb]{\smash{$v_{-r}$}}}%
\put(0,0){\includegraphics[width=\unitlength,page=2]{hyperbolic_sector.pdf}}%
\put(0.39438286,0.40043991){\color[rgb]{0,0,0}\makebox(0,0)[lb]{\smash{$\varphi$}}}%
\put(0.25357009,0.53395184){\color[rgb]{0,0,0}\makebox(0,0)[lb]{\smash{$R/2$}}}%
\put(0.35967099,0.21207239){\color[rgb]{0,0,0}\makebox(0,0)[lb]{\smash{$A_v$}}}%
\put(0.30398706,0.36311269){\color[rgb]{0,0,0}\makebox(0,0)[lb]{\smash{{\color{red}$\mathfrak{H}$}}}}%
\put(0.20426926,0.10207674){\color[rgb]{0,0,0}\makebox(0,0)[lb]{\smash{{\color{blue}$\mathfrak{F}(v_{-r})$}}}}%
\put(0.47088614,0.10214552){\color[rgb]{0,0,0}\makebox(0,0)[lb]{\smash{{\color{blue}$\mathfrak{F}(v_{r})$}}}}%
\put(0,0){\includegraphics[width=\unitlength,page=3]{hyperbolic_sector.pdf}}%
\put(0.4024499,0.65102187){\color[rgb]{0,0,0}\makebox(0,0)[lb]{\smash{}}}%
\put(0.39399511,0.64425801){\color[rgb]{0,0,0}\makebox(0,0)[lb]{\smash{$Q_{\varphi,R/2}$}}}%
\end{picture}%
\endgroup%
\caption{The Voronoi cell  $A_v$  based  at the vertex $v$.}\label{fig:Voronoi} 
\end{center}
\end{figure}

\begin{claim*}
Let $u_i$  be one of the Delaunay vertices  adjacent to $v$. If $u_i \notin \mathfrak{H}$ then
$$\mathcal{W}_r^R \subset     \mathfrak{F}(w_i) \subset   \mathfrak{F}(u_i). $$
\end{claim*}
\begin{proof}[Proof of Claim]
Consider a point $x \in \mathcal{W}_r^R$. We need  to show that $x \in \mathfrak{F}(w_i) $.  Note that $w_i$ and $x$ lie on opposite sides of the bi-infinite geodesic $\rho$. Both $w_i$ and $x$ belong to the following locally convex and contractible and hence \emph{convex} set
 $$B_{\widetilde{X}}(v,r) \cup \left( \mathfrak{F}(v_{-2r}) \cap \mathfrak{F}(v_{2r}) \cap  \mathfrak{H} \cap B_{\widetilde{X}}(v,R)\right).$$ 
 Therefore the intersection point $z$ of $\rho$ with   the geodesic arc  from $w_i$ to $x$    lies along the geodesic arc from $v_{-r}$ to $v_r$. This relies on the fact that $r \le R$ as follows from the assumptions and on the observation that 
 $\mathfrak{F}(v_{-2r}) \cap \mathfrak{F}_{2r} \cap \rho$ is   equal to the  geodesic segment from $v_{-r}$ to $v_{r}$.

Assume without loss of generality that  in fact the intersection point $z$ lies along the geodesic arc from $ v$ to $v_r$. Consider the two geodesic triangles 
$$T_1 =  \triangle(v,v_r,w_i) \quad \text{and} \quad  T_2 = \triangle(x,v_r,w_i).$$ The triangles $T_1$ and $T_2$  have no singular points in their interior. Therefore $T_1$ and $T_2$ are isometric   to their respective comparison hyperbolic triangles according to Corollary \ref{cor:convex subset embeds into H}.   
The triangle $T_1$ is isosceles with $d_{\widetilde{X}}(v,w_i) = d_{\widetilde{X}}(v,v_r)   = r$. Observe that the respective angles of $T_1$ and $T_2$ at the vertices $w_i$ and $v_r$ satisfy
$$ \measuredangle_{T_1} w_i \ge \measuredangle_{T_2}w_i
\quad  \text{and} \quad
\measuredangle_{T_1} v_r \le \measuredangle_{T_2}v_r.
$$
The hyperbolic law of sines \cite[\S7.12]{beardon2012geometry} implies   
 $$ d_{\widetilde{X}}(x,v_r) \le d_{\widetilde{X}}(x,w_i).$$
Finally since $x \in \mathfrak{F}(v_r)$   we have  that $d_{\widetilde{X}}(x,v) \le d_{\widetilde{X}}(x,v_r)$ and so $x \in \mathfrak{F}(w_i)$.
\end{proof}

Denote $I = \{ i \in \{1,\ldots,n\} \: : \: u_i \in \mathfrak{H} \}$ and $J = \{1,\ldots, n\} \setminus I$.  The above Claim implies that
$\mathcal{W}_r^R  \subset  \bigcap_{j \in J} \mathfrak{F}(u_j)$.
On the other hand  the assumptions of Lemma \ref{lem:hyperbolic geometry lemma} imply that  $B_{\widetilde{X}}(v,R/2) \subset \bigcap_{i \in I} \mathfrak{F}(u_i) $. Recall that the Voronoi cell in question $A_v$ is equal to  $ \bigcap_{i=1}^n \mathfrak{F}(u_i)$. Combining the above statements gives
$ \mathcal{W}_r^R \subset \mathfrak{H} \cap A_v$.
\end{proof}







\section{Cut graphs on branched covers}
\label{sec:cut graphs}

Let $S$ be a compact hyperbolic surface.  Consider an arbitrary  $*$-cover $X $ of $S$. To prove that $S$ is flexibly geometrically stable we need to   construct  a $*$-graph $\Gamma$ on $X$ so that the complement $X \setminus \Gamma$ embeds into some cover of $S$. In addition the  total length of $\Gamma$  should be controlled by the   branching degree of $X$ when both quantities are normalized relative to the size of $X$.  This motivates the following.


\begin{definition*}
\label{def:cut graphs}
Let $X$ be a $*$-cover of $S$. A \emph{$c$-cut graph}  for some   $c > 0$  is a $*$-graph $\Gamma$ on $X$ satisfying
\begin{enumerate}
	\item \label{it:Gamma contains singular locus}
	the singular set $s(X)$   is equal to the vertex set $\Gamma^{(0)}  $, 
	\item \label{it:angle of Gamma is at most pi} 
	if $e_1$ and $e_2$ are two edges  incident at the vertex $x \in \Gamma^{(0)}$ and consecutive in the cyclic order on the link of the graph $\Gamma$ at the vertex $x$ induced by its embedding in $X$ then the angle   between $ e_1$ and $e_2$ at $x$  is at most  $\pi$, and
\item 	\label{it:control of comninatorial length for Gamma}
	the total edge length $l(\Gamma)$   is bounded above by $ c \mathrm{Area}(X)$.
	\end{enumerate}
	A \emph{cut graph} on $X$ is a $c$-cut graph for some $c > 0$.
\end{definition*}
The above Condition $(\ref{it:angle of Gamma is at most pi})$   is equivalent to saying that every connected component of $X \setminus \Gamma$ has locally convex boundary.

The main goal of the current section is the following.
\begin{theorem}
\label{thm:existance of cut graphs}
For every $c > 0$ there is a $\delta = \delta(c) > 0$   such that any $*$-cover $X$ with $\beta(X) < \delta \mathrm{Area}(X)$ admits a $c $-cut graph.
\end{theorem}

The graph $\Gamma$ will be constructed by   considering the Delaunay graph  and then  carefully removing some of its edges until the graph-theoretical degree at each vertex  in roughly proportional to the local branching degree.
The function $c$ will depend on the topology and the metric of the surface $S$.

\begin{proof}[Proof of Theorem \ref{thm:existance of cut graphs}]
Let $ p : X \to S$ be a $*$-cover whose branching degree satisfies $\beta(X) <  \delta |X|$ for some sufficiently small $\delta > 0$ to be determined below. Recall that the branching degree $\beta(X)$ is defined to be $ \sum_{v \in s(X)} (d_v-1)$ where $d_v$ is the index of the point  $v$.

We may assume that  the singular set $s(X)$ is non-empty for otherwise the empty graph is a $c$-cut graph for any $c > 0$.
Consider the singular set $s(\widetilde{X}) = q^{-1}(s(X))$. 
Since $s(\widetilde{X}) \subset (p \circ q)^{-1}(*)$ this set is $r$-separated for some constant $r > 0$   depending only on the injectivity radius of the compact surface $S$.

Consider the family of Voronoi cells in $\widetilde{X}$  with respect to the vertex set $s(\widetilde{X})$. The interior of the every cell $A_{\widetilde{v}}$ with $\widetilde{v} \in s(\widetilde{X})$   embeds into $X$   according to Corollary \ref{cor:Voronoi embeds}. Denote $A_v = q(A_{\widetilde{v}})$ where $v \in s(X)$ and $\widetilde{v}$ is any vertex in $q^{-1}(v)$. The family  of the $A_v$'s with $v \in s(X)$ forms a tessellation of $X$ and satisfies
\begin{equation}
 \sum_{v \in s(X)} \mathrm{Area}(A_v) = \mathrm{Area}(X). 
\tag{$\mathrm{I}$}
\end{equation}

Recall that a half-space   in  $\widetilde{X}$ is the closure of a connected component of the complement of some bi-infinite geodesic line in $\widetilde{X}$. A \emph{half-cell} $\mathfrak{C}$   at the vertex  $v \in s(X)$ is $  q(A_{\widetilde{v}} \cap \mathfrak{H})$ for some half-space $\mathfrak{H}$ at any vertex $\widetilde{v} \in q^{-1}(v)$.

Let $\mathrm{D} $ be the geometric realization of the Delaunay graph  in $\widetilde{X}$ with respect to the vertex set $s(\widetilde{X})$. This notion is discussed  in Section \ref{sec:hyperbolic planes with singularity} above. Let $\mathrm{E}$ denote the projection $q(\mathrm{D})$ of the Delaunay graph  to the $*$-cover $X$. It follows from Proposition \ref{prop:Delaunay graph is embedded} that $\mathrm{E}$ is a $*$-graph in $X$. This $*$-graph  satisfies Condition (\ref{it:Gamma contains singular locus}) by its construction and Condition (\ref{it:angle of Gamma is at most pi}) according to Proposition \ref{prop:the angle between two edges of Delaunay graph is at most pi}. 
The required cut graph will be obtained by discarding some of the edges of      $\mathrm{E}$ while making sure Condition (\ref{it:angle of Gamma is at most pi}) continues to hold.


Let $L_v$ denote the set of edges of the graph $ \mathrm{E}$ incident at the vertex $v \in s(X)$. 
Our next step is to find a subset $M_v$ of $L_v$ of size $|M_v| \le  4d_v$ that continues to satisfy Condition $(\ref{it:angle of Gamma is at most pi})$ of cut graphs and such that 
\begin{equation}
 \varphi  \sum_{e \in M_v} \sinh^2(l(e)/4) \le 2 \mathrm{Area}(A_v) 
\tag{$\mathrm{II}$} \end{equation}
where $\varphi > 0$ is the constant    given in  Lemma \ref{lem:hyperbolic geometry lemma} with respect  to the parameter $r$.

Consider the subset  $K_v$  of $L_v$ consisting of all edges $e$ such that $$  2\varphi\sinh^2(l(e)/4) \le \mathrm{Area}(\mathfrak{C}_e  )$$ holds for \emph{some} half-cell $\mathfrak{C}_e$   at $v$ containing the edge $e$ in its interior.
 We claim that $K_v$ satisfies Condition $(\ref{it:angle of Gamma is at most pi})$ of cut graphs.  If this was not the case then there would exist a half-cell $\mathfrak{C} $  at $v$  which does  not contain any edge of $K_v$ in its interior.  However  at least one edge $e \in L_v$  is contained in $\mathfrak{C}$ and   every such edge is therefore not in $K_v$. Therefore   $  2\varphi\sinh^2(l(e)/4) > \mathrm{Area}(\mathfrak{C} )$ for every edge $e \in L_v$ contained in the half-cell $\mathfrak{C}$. This stands in contradiction to Lemma \ref{lem:hyperbolic geometry lemma}.
 
Take $M_v$ to be a  subset of $ K_v$ minimal with respect to containment  that still satisfies Condition $(\ref{it:angle of Gamma is at most pi})$ of cut graphs. The minimality of $M_v$ implies that  any half-cell $\mathfrak{C}$  at $v$  contains at most two edges from $M_v$. Since $A_v$ is the union of $2d_v$ half-cells at $v$  it follows that $|M_v| \le 4d_v$. Therefore any point $x$ of the Voronoi cell $A_v$   belongs to at most four half-cells from the family $\{\mathfrak{C}_e\}_{e\in M_v}$. Equation $(\mathrm{II})$ follows.


Consider the sub-$*$-graph $\Gamma$ of the $*$-graph $\mathrm{E}$ embedded in $X$ with vertex set $s(X)$ and edge set consisting  of $\bigcup_{v\in s(X)} M_v$.  Observe that Conditions   (\ref{it:Gamma contains singular locus}) and $(\ref{it:angle of Gamma is at most pi})$ of cut graphs hold true. 
It remains to bound the total length $l(\Gamma)$ from above as in Condition (\ref{it:control of comninatorial length for Gamma}) of cut graphs and to determine the precise value of $\delta$.

Let $\mathcal{E}$ denote the set of edges of $\Gamma$. Since $\frac{1}{2}\Sigma_{v \in V} |M_v| \le |\mathcal{E}| \le \Sigma_{v \in V} |M_v|$ and $2d_v \le |M_v| \le 4d_v$ for every $v \in V$    we have  that 
$$ \beta(X) \le  \beta(X) + |s(X)|    \le |\mathcal{E}| \le 4(\beta(X) + |s(X)|) \le 8\beta(X). $$
Making use of the area formula  $(\mathrm{I})$ and summing   equation $(\mathrm{II})$ over all vertices $v \in s(X)$  gives
$$  \varphi   \sum_{e \in \mathcal{E}}   \sinh^2(l(e)/4) \le  2\mathrm{Area}(X). $$
The function $\sinh^2(\cdot / 4)$ is convex. Jensen's inequality gives
$$ \varphi \sinh^2\left(\frac{l(\Gamma)}{4|\mathcal{E}|}\right) \le
 \frac{2\mathrm{Area}(X)}{|\mathcal{E}|}   $$
and after rearranging
$$ l(\Gamma) \le 4|\mathcal{E}| \arcsinh \left(\sqrt{\frac{2}{\varphi} \frac{\mathrm{Area}(X)}{|\mathcal{E}|}}\right).$$

To conclude the proof define the function $c : \delta \mapsto c_\delta $ by
$$ c_{\delta} = 32 \delta \arcsinh\left(\sqrt{ \frac{2}{\varphi \delta}}\right).$$
Note that $c $ is monotone and that  $\lim_{\delta \to 0} c_\delta = 0$. Finally observe that 
$$ l(\Gamma)   \le c_{\frac{ \beta(X)}{\mathrm{Area}(X)}} \mathrm{Area}(X) \le c_{\delta} \mathrm{Area}(X). $$
Therefore $\Gamma$ is a $c$-cut graph provided $\delta$ is sufficiently small so that $c_\delta < c$.
\end{proof}

\section{Boundaries and non-separating closed curves}

Let $S$ be a compact hyperbolic surface with a non-separating  family of disjoint simple curves. We study a necessary and a sufficient condition due to Walter D. Neumann \cite[Lemma 3.2]{neumann2001immersed} for the existence of a cover for $S$ such that the preimages of the given family of curves have prescribed degrees.

\subsection*{$1$-manifolds} A  connected  closed  \emph{$1$-manifold} is homeomorphic to $S^1$. A closed compact $1$-manifold is a disjoint union of finitely many homeomorphic copies of $S^1$.    The set of the connected components of a $1$-manifold $B$ is denoted $\pi_0(B)$.   $1$-manifolds are orientable and admit $2^{|\pi_0(B)|}$ orientations. 

The boundary $\partial F$ of a compact surface $F$ is a closed $1$-manifold. If the surface $F$ is oriented then $\partial F$ inherits an induced orientation.

\subsection*{The Neumann lemma} 
We use $\chi(F)$ to denote the Euler characteristic of the surface $F$.  Fix a degree $N \in \NN$.
\editA{

\begin{lemma}[{\cite[Lemma 3.2]{neumann2001immersed}}]
\label{lem:Neumann}
Let $F$ be a compact connected orientable surface of positive genus. Let $ p : B \to \partial F$ be a covering map for some closed $1$-manifold $B$. 
The following two statements are equivalent:
\begin{enumerate}
    \item There exist a connected surface $R$ with $\partial R \cong B$ and a covering map $r:R \to F$  of degree $N$ with $r_{|B} = p$.
    \item $|\pi_0(B)|$ has the same parity as  $\chi(F) N$ and for every connected component $\gamma$ of $\partial F$  the total degree of the covering map $p$ on ${p^{-1}(\gamma)}$   is   $N$.  
\end{enumerate}
The implication (1)$\implies$(2) holds true even without assuming that  $F$ has positive genus, or that $R$ is connected.
\end{lemma}




We remark that in  \cite{neumann2001immersed} the equivalence of (1) and (2) is proved under the assumption of positive genus. However, the proof of the direction (1)$\implies$(2) clearly applies in any genus.
}
%
%

We obtain the following    consequence of Lemma \ref{lem:Neumann}. It will be used towards the proof of Theorem \ref{thm:variant of LERF} presented  in Section \ref{sec:capping off}. 

\begin{cor}
  \label{cor:prescribed boundary}
Let $S$ be a closed  hyperbolic surface. Let $p : B \to A$ be a covering map of closed compact and oriented $1$-manifolds where $A$ is embedded and is non-separating in $S$.
  Then there exist
  \begin{itemize}
  \item a covering map $r : C \to S$ and
  \item \editA{an oriented subsurface $R$ of $C$ such that $\partial R \cong B$ as an oriented $1$-manifold with the induced orientation and with $r_{|B} = p$}
  \end{itemize}
  if and only if 
$|\pi_0(B)|$ is even  and $ B = B^+ \cup B^-$ where
\begin{itemize}
\item $p^+ = p_{|B^+}$ is orientation preserving,  
\item $p^- = p_{|B^-}$ is orientation reversing and  
\item for every connected component $\alpha$ of $A$ the total \editA{covering} degrees of $p^+$ and of $p^-$ on $p^{-1}(\alpha)$ are equal to each other.
\end{itemize}

If the latter conditions are satisfied then the cover  $r$ can be taken of degree equal to \editA{four times} the maximum of the total degree of $p$ over any connected component of $A$.
\end{cor}

We   consider a   topological reduction needed to apply Neumann's lemma towards the proof of  Corollary \ref{cor:prescribed boundary}. The main issue is to allow for the degree of $p$ to vary over the different connected components $\alpha$ of $A$.

 Let $\bar{S}$ denote   the completion of $S \setminus A$. In particular $\bar{S}$ is a compact connected orientable surface with boundary. Consider the       quotient map $f : \bar{S} \to S$ with $$f(\partial \bar{S}) = A.$$
 
  Choose   orientations for the two surfaces $S$ and $\bar{S}$ such that   $f$ is orientation preserving on the interior of $\bar{S}$.
Equip   $\partial \bar{S}$ with the  boundary orientation induced from that of $\bar{S}$.
 In particular   
 $$\partial \bar{S} = \bar{A}^+ \amalg \bar{A}^-$$
 in such a way that the restriction of $f$ to $\bar{A}^+$ and to $\bar{A}^-$ is, respectively,  orientation preserving and orientation reversing. 
 
Observe that if  $q : \bar{C} \to \bar{S}$ is  any orientation preserving   cover and $\bar{E} \subset \partial \bar{C}$ is any connected closed $1$-manifold then the composition $f \circ q : \bar{E} \to A$ is orientation preserving if and only if $q(\bar{E}) \subset \bar{A}^+$. 
Finally let $\tau : \partial \bar{S} \to \partial \bar{S}$  be the homeomorphism determined  by 
$$\tau(\bar{A}^+) = \bar{A}^- \quad \text{and} \quad f \circ \tau = f.$$

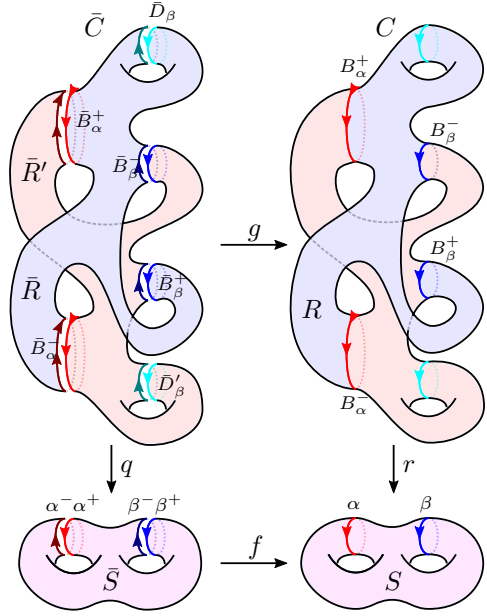
\begin{figure}
	\begin{center}
		\def\svgwidth{0.5\textwidth}
\begingroup%
\makeatletter%
\providecommand\color[2][]{%
	\errmessage{(Inkscape) Color is used for the text in Inkscape, but the package 'color.sty' is not loaded}%
	\renewcommand\color[2][]{}%
}%
\providecommand\transparent[1]{%
	\errmessage{(Inkscape) Transparency is used (non-zero) for the text in Inkscape, but the package 'transparent.sty' is not loaded}%
	\renewcommand\transparent[1]{}%
}%
\providecommand\rotatebox[2]{#2}%
\ifx\svgwidth\undefined%
\setlength{\unitlength}{196.01929912bp}%
\ifx\svgscale\undefined%
\relax%
\else%
\setlength{\unitlength}{\unitlength * \real{\svgscale}}%
\fi%
\else%
\setlength{\unitlength}{\svgwidth}%
\fi%
\global\let\svgwidth\undefined%
\global\let\svgscale\undefined%
\makeatother%
\begin{picture}(1,1.23309083)%
\put(0,0){\includegraphics[width=\unitlength,page=1]{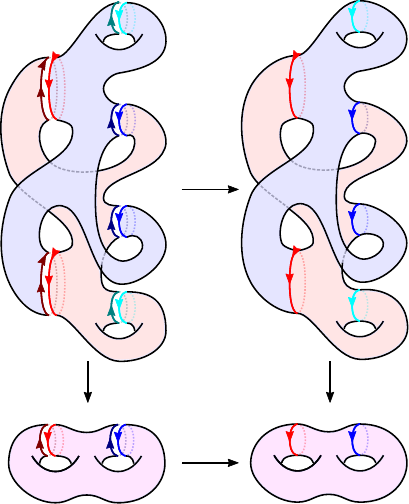}}%
\put(0.70947401,0.21391257){\color[rgb]{0,0,0}\makebox(0,0)[lb]{\smash{{\scriptsize$\alpha$}}}}%
\put(0.07657135,0.21043257){\color[rgb]{0,0,0}\makebox(0,0)[lb]{\smash{{\scriptsize$\alpha^-$}}}}%
\put(0.13396366,0.21043257){\color[rgb]{0,0,0}\makebox(0,0)[lb]{\smash{{\scriptsize$\alpha^+$}}}}%
\put(0.69338406,0.42485605){\color[rgb]{0,0,0}\makebox(0,0)[lb]{\smash{{\scriptsize$B^-_\alpha$}}}}%
\put(0.69338406,1.13269459){\color[rgb]{0,0,0}\makebox(0,0)[lb]{\smash{{\scriptsize$B^+_\alpha$}}}}%
\put(0.86252021,0.21391279){\color[rgb]{0,0,0}\makebox(0,0)[lb]{\smash{{\scriptsize$\beta$}}}}%
\put(0.24874827,0.21043257){\color[rgb]{0,0,0}\makebox(0,0)[lb]{\smash{{\scriptsize$\beta^-$}}}}%
\put(0.30614055,0.21043257){\color[rgb]{0,0,0}\makebox(0,0)[lb]{\smash{{\scriptsize$\beta^+$}}}}%
\put(0.88257439,0.74983004){\color[rgb]{0,0,0}\makebox(0,0)[lb]{\smash{{\scriptsize$B^+_{\beta}$}}}}%
\put(0.88257439,0.99853012){\color[rgb]{0,0,0}\makebox(0,0)[lb]{\smash{{\scriptsize$B^-_{\beta}$}}}}%
\put(0.29163791,1.24747943){\color[rgb]{0,0,0}\makebox(0,0)[lb]{\smash{{\scriptsize  $\bar{D}_{\beta}$}}}}%
\put(0.31076869,0.46311796){\color[rgb]{0,0,0}\makebox(0,0)[lb]{\smash{{\scriptsize $\bar{D}'_{\beta}$}}}}%
\put(0.02380713,0.61616382){\color[rgb]{0,0,0}\makebox(0,0)[lb]{\smash{}}}%
\put(0.02380713,0.65442561){\color[rgb]{0,0,0}\makebox(0,0)[lb]{\smash{$\bar{R}$}}}%
\put(0.02380713,0.90312569){\color[rgb]{0,0,0}\makebox(0,0)[lb]{\smash{$\bar{R}'$}}}%
\put(0.61474364,0.61591469){\color[rgb]{0,0,0}\makebox(0,0)[lb]{\smash{$R$}}}%
\put(0.84643025,0.31007165){\color[rgb]{0,0,0}\makebox(0,0)[lb]{\smash{}}}%
\put(0.2342456,0.29094103){\color[rgb]{0,0,0}\makebox(0,0)[lb]{\smash{$q$}}}%
\put(0.78903795,0.04224096){\color[rgb]{0,0,0}\makebox(0,0)[lb]{\smash{$S$}}}%
\put(0.19598405,0.04224096){\color[rgb]{0,0,0}\makebox(0,0)[lb]{\smash{$\bar{S}$}}}%
\put(0.82729941,0.29094103){\color[rgb]{0,0,0}\makebox(0,0)[lb]{\smash{$r$}}}%
\put(0.50207635,0.11876411){\color[rgb]{0,0,0}\makebox(0,0)[lb]{\smash{$f$}}}%
\put(0.7699071,1.20921785){\color[rgb]{0,0,0}\makebox(0,0)[lb]{\smash{$C$}}}%
\put(0.1577225,1.20921785){\color[rgb]{0,0,0}\makebox(0,0)[lb]{\smash{$\bar{C}$}}}%
\put(0.50207635,0.78834107){\color[rgb]{0,0,0}\makebox(0,0)[lb]{\smash{$g$}}}%
\put(0.0429379,0.539641){\color[rgb]{0,0,0}\makebox(0,0)[lb]{\smash{{\scriptsize $\bar{B}^-_\alpha$}}}}%
\put(0.13859174,1.01791019){\color[rgb]{0,0,0}\makebox(0,0)[lb]{\smash{{\scriptsize $\bar{B}^+_\alpha$}}}}%
\put(0.31076869,0.67355646){\color[rgb]{0,0,0}\makebox(0,0)[lb]{\smash{{\scriptsize $\bar{B}^+_{\beta}$}}}}%
\put(0.21511481,0.92225631){\color[rgb]{0,0,0}\makebox(0,0)[lb]{\smash{{\scriptsize $\bar{B}^-_{\beta}$}}}}%
\end{picture}%
\endgroup%
		\caption{ The maps $r$ and $q$ are covers, the maps $f$ and $g$  are quotients, and the diagram is commutative.
					In this example $A = \{\alpha, \beta\}$ and $\bar{A}^\pm = \{\alpha^\pm, \beta^\pm\}$. The curve $\alpha$ satisfies $N_\alpha = M$ and the curve $\beta$ satisfies $N_\beta < M$.}
		\label{fig:fancy drawing}
	\end{center}

\end{figure}
\vspace{10pt}
\begin{center}
\textit{The main ideas of the following proof are summarized in Figure \ref{fig:fancy drawing}.}
\end{center}
\begin{proof}[Proof of Corollary \ref{cor:prescribed boundary}]

%

\emph{In one direction.} Consider a cover $r :C \to S$   and a  subsurface $R$ of $C$  with $\partial R \cong B $   as in the statement of Corollary \ref{cor:prescribed boundary}.  

Let $\bar{C}$ denote the compact surface with boundary obtained as the completion of $C \setminus r^{-1}(A)$. Note that $\bar{C}$ is  disconnected.   There is a natural quotient map $g : \bar{C} \to C$ as well as a covering map $q : \bar{C} \to \bar{S}$ satisfying 
$$f \circ q = r\circ g. $$
 
Denote $\bar{E} = \partial \bar{C}$ so that   $\bar{E}$ is a compact closed $1$-manifold  and $f \circ q  : \bar{E} \to A$ is a covering map of $1$-manifolds. In particular $\partial R \cong B$ is a sub-$1$-manifold of $g(\bar{E})$.
 

Let $\bar{R}$ be the subsurface of $\bar{C}$ corresponding to $R$.
We may apply    Lemma \ref{lem:Neumann}
to the restricted covering map $q : \bar{R} \to \bar{S}$ and to the boundary $\partial \bar{R} = \bar{R} \cap \bar{E}$. The result follows in this direction by letting 
$$ \text{$B^+ = g(q^{-1}(\bar{A}^+)) \cap B$ \quad and \quad $B^- = g(q^{-1}(\bar{A}^-)) \cap B$}. $$

Observe that the restriction of $r$ to $B^+$ and to $B^-$ is orientation preserving and orientation reversing respectively. The connected components of the $1$-manifold $\bar{D} =  \bar{E} \setminus g^{-1}(\partial R)$ occur in matching pairs $\delta^+$ and $\delta^-$ with $g(\delta^+) = g(\delta^-)$.   Therefore the total degrees  of $p^+ = p_{|B^+}$ and of $p^- = p_{|B^-}$ agree over every connected component of $A$,  and $|\pi_0(B)|$ is even.

\emph{In the other direction.}  
\editA{Assume first that the surface $\bar S$ has positive genus.}
Assume that $|\pi_0(B )|$ is even and that there exists a decomposition $B = B^+ \cup B^-$  as in the statement. 
\editA{For each connected component $\alpha$ of $A$, 
let $N_\alpha \in \mathbb{N}$ denote the common total degree of $p^+$ and of $p^-$ over $\alpha$.} Denote $M = \max N_\alpha$ taken over all such $\alpha$'s.

\editA{
We first enlarge $B$ to a 1-manifold $\bar E$ with a covering map $t:\bar E\to \partial\bar S$ such that over every component of $A$ the degree is the same. If $\alpha$ is a component of $A$ satisfying $N_\alpha < M$, let $D_\alpha$ denote the disjoint union of two circles $D_\alpha^+$ and $D_\alpha^-$, equipped with covering maps $t_\alpha^\pm:D_\alpha^\pm \to \bar\alpha^\pm$ of degree $M-N_\alpha$ (where $\bar\alpha^\pm$ are the two components of $f^{-1}(\alpha)$). 
Denote $D = D^+ \cup D^-$ where $D^\pm = \bigcup_\alpha D^\pm_{\alpha}$ with the union being taken over all components $\alpha$ of $A$ satisfying $N_\alpha< M$. 
 
Let $\bar E$ be the disjoint union of $B$ and $D$, and let $t:\bar E\to \partial\bar S$ be the covering map defined by $t_\alpha^\pm$ on each component $D_\alpha^\pm$ of $D$, and on $B$ by
$$t(B^+ ) = \bar A^+, \quad t(B^-) = \bar A^-, \quad \text{and} \quad f \circ t = p.$$

}

The covering map $ t$ from $\bar{E}$ to $\partial \bar{S}$  has total degree $M$ over every connected component of $\partial \bar{S}$. We may apply Lemma \ref{lem:Neumann} with respect to the surface $\bar{S}$ and the covering map $t$ to obtain  a connected   surface $\bar{R}$ with $\partial \bar{R} \cong \bar{E}$ and a covering map $q  : \bar{R} \to \bar{S}$ of degree $M$     with $q_{|\bar{E}} =  t$. \editA{Note that the parity hypothesis of Lemma \ref{lem:Neumann} is satisfied since both $|\pi_0(B)|$ and $\chi(\bar S)$ are even.}

Let \editA{$\bar{E}' = B' \cup D'$} be a homeomorphic copy of the $1$-manifold $\bar{E}$ considered with the twisted covering map $t' = \tau \circ t$. This covering map $t'$ has degree $M$. We may apply Lemma \ref{lem:Neumann} again with respect to the covering map $t'$ to obtain a corresponding surface cover $q' : \bar{R}' \to \bar{S}$.

Denote
$ \bar{C}= \bar{R}\cup\bar{R}'$. Allowing for a slight abuse of notation, let $q : \bar{C}\to \bar{S}$ denote the union of the two covering maps $q : \bar{R}\to \bar{S}$ and $q' : \bar{R} \to \bar{S}$.
There is  a quotient map $g : \bar{C} \to C$ defined by identifying boundary   components of $\bar{C}$ as follows. \editA{Identify $D^+$ with $D^-$ and  $D'^+$ with $D'^-$. Moreover identify $B^+$ with $B'^-$ and $B'^+$ with $B^-$}. The identification is performed respecting the covering map $q$. In particular, there exists a resulting covering map $r : C \to S$ with $r \circ g = f \circ q$. 

The cover $C$ has degree $2M$, which is equal to the maximal total degree of $p$. Let $R = g(\bar{R})$. This completes the proof of this direction under the assumption that $\bar S$ has positive genus.

\editA{Now, assume that the surface $\bar S$ has genus 0. Let $S'$ be the surface obtained by splitting $S$ along a curve $\alpha\in A$ and gluing two copies of the resulting surface along their boundaries. Let $q':S'\to S$ be the  obvious covering map mapping the two copies back to $S$.  Each curve in $A$ lifts to two curves in $S'$. Let $A' \subset q'^{-1}(A)$ be a collection of curves containing exactly one lift for each curve in $A$. Note that $q'_{|A'}:A'\to A$ is a homeomorphism as each curve is covered with degree 1. The completion $\bar S'$ of $S'\setminus A'$ is connected and has positive genus.

Denote by $p':B\to A'$ the map $p'=(q'_{|A'})^{-1} \circ p$.
The surface $\bar S'$ has positive genus and so the previous argument can be applied to get a covering map $r' : C \to S'$ of degree $2M$ and an oriented subsurface $R$ of $C$ such that $\partial R \cong B$ and $r'_{|B} = p'$. The composition $r = q' \circ r' : C \to S$ is as required, and has degree $4M$.}
\end{proof}

\section{Quantitatively capping off surfaces with boundary}
\label{sec:capping off}

Let $S$ be a closed hyperbolic surface of genus  $g \ge 2$. The goal of the current section is to prove Theorem \ref{thm:variant of LERF}, restated below for the reader's convenience.

\begin{theorem*}
 Let $R$ be a   surface with boundary which is isometrically embedded in some cover of $S$.
 If the boundary   $\partial R$  is  locally convex then  $R$ can be isometrically embedded in a cover $Q$ of $S$ making Diagram (\ref{eq:commute R}) commute and  such that
$$\mathrm{Area}(Q \setminus R ) \le  b_S l(\partial R)$$
where $b_S > 0$ is a constant depending only on $S$.
\end{theorem*}

We will deal with the   problem of capping-off a boundary component of $R$ in a controlled way by working in a  certain combinatorial framework. More precisely we will consider surfaces tessellated by isometric copies of a   particularly nice   fundamental domain. The generic   situation as in the above theorem can be easily reduced to this combinatorial framework.

\subsection*{Surfaces with locally convex boundary}
We point out the following consequence of Lemma \ref{lem:lifting local isometry} and of the Cartan--Hadamard theorem.
\begin{lemma}
\label{lem:local isometry and locally convex boundary}
Let $R$ be a surface with boundary. Then the boundary $\partial R$ is locally convex and $R$ is isometrically embedded in some cover of $S$ if and only if $R$ admits a local isometry to $S$.
\end{lemma}

Here we use the term \emph{local isometry} in the   metric space   rather than   Riemannian geometry sense.

\begin{proof}[Proof of Lemma \ref{lem:local isometry and locally convex boundary}]
If $R$ is isometrically embedded in some cover $ p : C \to S$ and its boundary $\partial R$ is locally convex then the restriction of $p$ to $R$ is a local isometry. 

Conversely assume that $R$ admits a local isometry $f$ into $S$. This implies that the boundary $\partial R$ is locally convex. The map $f_* : \pi_1(R) \to \pi_1(C)$ is injective and $f$ lifts to an isometric embedding $F : \widetilde{R} \to \widetilde{S} \cong \mathbb{H}$ of the universal covers according to Lemma \ref{lem:lifting local isometry}. The map $F$ descends to an isometric embedding of $R$ into the cover of $S$ that corresponds to the subgroup $f_* \pi_1(R) \le \pi_1(S)$.
\end{proof}

\subsection*{The family of closed geodesic curves $\lambda_0,\ldots,\lambda_{2g-1}$}

\begin{lemma}
\label{lem:construction of a RA fundamental domain}
The surface $S$ admits a family $\lambda_0,\ldots,\lambda_{2g-1}$  of non-separating simple  closed geodesic curves such that 
\begin{enumerate}
\item the   geodesics curves $\lambda_i$ and $\lambda_j$ with $i \neq j$ are disjoint unless $|i-j| = 1$  in which case they intersect at a single point, and
\item the complement $S \setminus (\lambda_0 \cup \cdots \cup \lambda_{2g-1})$ is a topological disc.
\end{enumerate}
\end{lemma}
\begin{proof}
Consider a family $\gamma_0,\ldots,\gamma_{2g-1}$ of non-separating simple closed curves positioned on the surface $S$ as described in Figure \ref{fig:identification on P} below. 

\begin{figure}[ht!]
\begin{center}
	\def\svgwidth{0.55 \textwidth}
\begingroup%
\makeatletter%
\providecommand\color[2][]{%
	\errmessage{(Inkscape) Color is used for the text in Inkscape, but the package 'color.sty' is not loaded}%
	\renewcommand\color[2][]{}%
}%
\providecommand\transparent[1]{%
	\errmessage{(Inkscape) Transparency is used (non-zero) for the text in Inkscape, but the package 'transparent.sty' is not loaded}%
	\renewcommand\transparent[1]{}%
}%
\providecommand\rotatebox[2]{#2}%
\ifx\svgwidth\undefined%
\setlength{\unitlength}{166.19392695bp}%
\ifx\svgscale\undefined%
\relax%
\else%
\setlength{\unitlength}{\unitlength * \real{\svgscale}}%
\fi%
\else%
\setlength{\unitlength}{\svgwidth}%
\fi%
\global\let\svgwidth\undefined%
\global\let\svgscale\undefined%
\makeatother%
\begin{picture}(1,0.38286455)%
\put(0,0){\includegraphics[width=\unitlength,page=1]{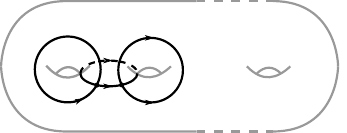}}%
\put(0.18022734,0.04809204){\color[rgb]{0,0,0}\makebox(0,0)[lb]{\smash{$e_0$}}}%
\put(0.30199753,0.10352897){\color[rgb]{0,0,0}\makebox(0,0)[lb]{\smash{$e_1$}}}%
\put(0.29724717,0.22893977){\color[rgb]{0,0,0}\makebox(0,0)[lb]{\smash{$e'_1$}}}%
\put(0.42630392,0.05148698){\color[rgb]{0,0,0}\makebox(0,0)[lb]{\smash{$e_2$}}}%
\put(0.42348442,0.29934255){\color[rgb]{0,0,0}\makebox(0,0)[lb]{\smash{$e'_2$}}}%
\put(0.54975954,0.10195753){\color[rgb]{0,0,0}\makebox(0,0)[lb]{\smash{$e_3$}}}%
\put(0.54400003,0.23311931){\color[rgb]{0,0,0}\makebox(0,0)[lb]{\smash{$e'_3$}}}%
\put(0.86641022,0.10179435){\color[rgb]{0,0,0}\makebox(0,0)[lb]{\smash{$e_{2g-1}$}}}%
\put(0,0){\includegraphics[width=\unitlength,page=2]{RA_surface.pdf}}%
\end{picture}%
\endgroup%
  \caption{The simple closed curves $\gamma_0,\ldots,\gamma_{2g-1}$ are given by  $\gamma_0 = e_0$, $\gamma_{2g-1} = e_{2g-1}$ and $\gamma_i = e_i \bar{e} '_i$ for every other $i$.}
  \label{fig:identification on P}
\end{center}
\end{figure}
Let $\lambda_i$ be the geodesic representative of the simple closed curve $\gamma_i$ for every $i \in \{0,\ldots,2g-1\}$. Since geodesic curves are in minimal position we have that
$$| \lambda_i \cap \lambda_j | = \begin{cases} 
1, & \text{if $ |i-j| = 1$} \\
0, & \text{if $ |i-j| > 1$} \\
\end{cases} $$
as required for Statement (1) of the Lemma.  The completion $\mathcal{P}$ of the complement $S \setminus (\lambda_0 \cup \cdots \cup \lambda_{2g-1})$ is a connected surface with boundary admitting $4(2g-1)$ vertices, $2 + 2 + 4(2g-2)=8g-4$ edges  and a single face. The Euler characteristic of $\mathcal{P}$ is  equal to
$$ \chi(\mathcal{P}) = 4(2g-1) - (8g-4) + 1 = 1.$$
Therefore $\mathcal{P}$ is a topological disc as required for Statement (2).
\end{proof}

Let $\mathcal{P}$ denote the completion of the complement $S \setminus (\lambda_0 \cup \cdots \cup \lambda_{2g-1})$. Therefore  $\mathcal{P}$ is compact convex hyperbolic $(8g-4)$-sided polygon. Moreover $\mathcal{P}$ is isometric to a fundamental domain for the action of the fundamental group $\pi_1(S)$ on the hyperbolic plane. We introduce the following notations for the edges of $\mathcal{P}$. See Figure \ref{fig:identification on P}.
\begin{itemize}
\item  $e_0$ and $\bar{e}_0$ are the two geodesic sides of $\mathcal{P}$ that correspond to the curve $\lambda_0$,
\item $e_{2g-1}$ and $\bar{e}_{2g-1}$ are the two geodesic sides of $\mathcal{P}$ that correspond to the curve $\lambda_{2g-1}$, and
\item $e_i,e'_i,\bar{e}_i$ and $\bar{e}'_i$ are  the four sides of $\mathcal{P}$ that correspond to the curve $\lambda_i$  for every other $i \in \{1,\ldots, 2g-2\}$.
\end{itemize}
The surface $S$ can be recovered by identifying every geodesic side $e$ of the fundamental polygon $\mathcal{P}$ with the  geodesic side of opposite orientation   denoted $\bar{e}$.

\medskip

\emph{We will keep the above notations   throughout the remainder of Section \ref{sec:capping off}.}


\subsection*{Tessellations and $\mathcal{P}$-surfaces}

Let $\mathcal{T}(C,\mathcal{P})$ denote the induced tessellation of a given cover $C$ of the surface $S$ by isometric copies of the fundamental polygon $\mathcal{P}$.

\editA{
The link of every vertex of the tessellation $\mathcal{T}(C,\mathcal{P})$ is a cycle graph of size four. In other words,  locally, exactly four polygons of $\mathcal{T}(C,\mathcal{P})$ meet at every vertex.} Another crucial property of the tessellation $\mathcal{T}(C,\mathcal{P})$ is that the concatenation of any pair of edges incident at a given vertex $v$ and \emph{not} consecutive in the cyclic ordering determined by the  link of $v$ is a local geodesic.   Such a concatenation of two edges is, up to edge orientation, of the form $e_0 e_0, e_{2g-1} e_{2g-1}, e_i e'_i$ or $e'_i e_i$ for some $i \in \{1,\ldots,2g-2\}$.

\begin{definition*}
 A \emph{$\mathcal{P}$-surface} $R$ is a subsurface of some  cover $C$ of the surface $S$   tiled by  polygons from the tessellation $\mathcal{T}(C,\mathcal{P})$. The induced tessellation of $R$ will be denoted $\mathcal{T}(R,\mathcal{P})$.
\end{definition*}

 
We remark that a $\mathcal{P}$-surface could in general be disconnected.





\begin{prop}
\label{prop:sides match up}
Let $R$ be a $\mathcal{P}$-surface. Then every geodesic arc $ e $ of $\partial \mathcal{P}$ appears along the boundary $\partial R$ the same number of times with each orientation $e$ and $\overline{e}$.
\end{prop}
\begin{proof}
Every geodesic arc $e$ appears along the boundary of the polygon $\mathcal{P}$ once with each orientation. Moreover in the tessellation of $R$ by isometric copies of $\mathcal{P}$ every interior edge is accounted for once with each orientation. The result follows.
\end{proof}

Given a  $\mathcal{P}$-surface $R$ it follows from the properties of the tessellation $\mathcal{T}(R,\mathcal{P})$ discussed above   that $\partial R$ is locally convex if and only the link of any vertex $v$ of   $\mathcal{T}(R,\mathcal{P})$ that lies on the boundary $\partial R$ is \editA{a path graph of length at most two. In other words, locally, at most two polygons meet at a vertex in the boundary,  forming together an angle of at most $\pi$.} This observation motivates the following.

\begin{prop}
	\label{prop:identifying same edges}
	Let $R$ be a $\mathcal{P}$-surface with locally convex boundary. Let $\alpha$ and $\beta$ be   two  totally geodesic closed curves/maximal geodesic arcs    along the boundary $\partial R$   that map to the same geodesic curve/arc  on $S$ with opposite orientations. Then the surface $R'$ obtained by identifying $\alpha$ and $\beta$ is  a $\mathcal{P}$-surface with locally convex boundary.
\end{prop}
\begin{proof}
The fact that $R'$ has a locally   convex boundary is clear in the case that $\alpha$ and $\beta$ are totally geodesic boundary curves.

Consider the case  where  $\alpha$ and $\beta$ are maximal geodesic arcs.  The vertices representing the end-points of the arcs $\alpha$ and $\beta$ in the tessellation $\mathcal{T}(R,\mathcal{P})$ have links of size one in $R$.  Therefore the boundary components of $R'$ that have been modified by the identification of $\alpha$ and $\beta$ have links of size at most two in $R$ and so remain locally convex as required.

The $\mathcal{P}$-surface structure of $R$ gives rise to   a local isometry $f : R \to S$. It is compatible with the identification of $\alpha$ and $\beta$ and descends to a local isometry $f' : R' \to S$.  It follows from Lemma \ref{lem:local isometry and locally convex boundary} that  $R'$ isometrically embeds into some cover  of $S$.  It is clear that $R'$ is tiled by polygons from the tessellation $\mathcal{T}(C,\mathcal{P})$. Therefore $R'$ is  a $\mathcal{P}$-surface.
\end{proof}

The notion of  $\mathcal{P}$-surfaces is useful for our purposes since a surface with   boundary as in Theorem \ref{thm:variant of LERF} can always be embedded inside a $\mathcal{P}$-surface in an efficient way.


\begin{prop}
\label{prop:cut surface is deformation retract of R surface}
Let $R$ be a  surface with boundary which is isometrically embedded in some cover of $S$. If the boundary   $\partial R$ is  locally convex then  $R$ can be isometrically embedded in a  $\mathcal{P}$-surface $Q$ with locally convex boundary making Diagram (\ref{eq:commute R}) commute and such that
$$\mathrm{Area}(Q \setminus R) \le  d_{S}  l(\partial R) \quad \text{and} \quad l(\partial Q) \le d_{S}  l(\partial R)$$
where $d = d_{S}  > 0$ is a constant depending only on the surface $S$. 

Moreover we may assume, if desired, that any geodesic subarc of $\partial Q$ has at most two edges and that no boundary component of $\partial Q$ is totally geodesic.
\end{prop}

The following proof is inspired by and relies on ideas  of Patel's work \cite{patel2014theorem}.

\begin{proof}[Proof of Proposition \ref{prop:cut surface is deformation retract of R surface}]
We assume that $R$ has non-empty boundary for otherwise there is nothing to prove.
 Since $R$ is embedded in some cover $C$ of $S$ we may identify $\pi_1(R)$ with a certain infinite index subgroup of $\pi_1(S)$. Let $R'$ be the cover of $S$ corresponding to that subgroup. In particular $R$ is an embedded subsurface of $R'$ and every connected component of $R' \setminus R$ retracts to a boundary component of $\partial R$. Let $Q'$ be the $\mathcal{P}$-surface consisting of all polygons in $\mathcal{T}(R',\mathcal{P})$ that intersect $R$ non-trivially.

The boundary of the $\mathcal{P}$-surface $Q'$ need not yet be locally convex\editA{, as it may have \emph{bad} vertices whose links are paths of length 3. As explained in \cite[Theorem 3.1]{patel2014theorem}, since $R$ is locally convex, there are no consecutive bad vertices in $\partial Q'$. One can now form the locally convex $\mathcal{P}$-surface $Q''$ by attaching polygons along the geodesics subarcs of $\partial Q'$ emanating from bad vertices. For the full details we refer the reader to 
Patel's argument in \cite[Theorem 3.1]{patel2014theorem}.}

\begin{remark*}
We emphasize that, at least formally speaking, there are a few minor differences between our situation and that of  \cite{patel2014theorem}. Namely
\begin{itemize}
	\item  A right-angled pentagon is used in \cite{patel2014theorem} while our polygon $\mathcal{P}$ is $(8g-4)$-sided and is not assumed to be right-angled.\footnote{\editA{In fact $\mathcal{P}$ can be chosen to be right-angled in our case as well. However, making this choice is not essential to the proof.}} 
	\item  The quotient of the hyperbolic plane by the action of a single hyperbolic element is considered  in the context of \cite[Theorem 3.1]{patel2014theorem} while we consider locally convex boundary components of  $\mathcal{P}$-surfaces.
\end{itemize}
Patel's arguments rely only on certain properties of   links in the tessellation $\mathcal{T}(R',P)$ that were mentioned above and hold true in our case as well. Namely, all links have size four, and the concatenation of every two non-consecutive edges is a local geodesic. The same proof goes through.
\end{remark*}

\editA{We conclude that the  $\mathcal{P}$-surface $Q''$ has   locally convex boundary and satisfies $Q ' \subset Q'' \subset R'$.}
Proceeding with the proof of Proposition \ref{prop:cut surface is deformation retract of R surface}, the boundary of the $\mathcal{P}$-surface $Q''$  may contain geodesic arcs longer than two edges or totally geodesic boundary components. If desired, this can be overcome simply by  attaching an additional layer of polygons to $Q''$. Namely, let $Q$ consist of all polygons in $\mathcal{T}(R',\mathcal{P})$ that intersect $Q''$ or its boundary non-trivially. The boundary of $Q$ remains     locally convex   and satisfies the requirements as in the statement. 


 An area estimate analogous to \cite[Theorem 4.3]{patel2014theorem} shows that $\mathrm{Area}(Q' \setminus R) \le  d'_{S } l(\partial R)$ for some constant $d'_S > 0$. The total boundary length of $Q'$ is bounded above by the number of  polygons meeting $Q' \setminus R$ times the perimeter of the polygon $\mathcal{P}$. In particular   $l(\partial Q') \le d''_{S} l(\partial R)$ for some other constant $d''_{S} > 0$. Repeating this argument twice for the pair of $\mathcal{P}$-surfaces $Q''$ and $Q$ gives the required linear upper bounds on $\mathrm{Area}(Q\setminus R)$ and $l(\partial Q)$ in terms of some positive constant $d_S > 0$.
\end{proof}

\subsection*{Capping off boundary components}

We now have all the required machinery to complete the proof of  Theorem \ref{thm:variant of LERF} of the introduction.

\begin{proof}[Proof of Theorem \ref{thm:variant of LERF}]
Let $R$ be a  surface with locally convex boundary which is embedded in some cover of $S$. Making use of Proposition \ref{prop:cut surface is deformation retract of R surface} it is possible to find a $\mathcal{P}$-surface $Q_1$ containing an embedded copy of $R$   such that $\mathrm{Area}(Q_1 \setminus R)$ as well as $ l(\partial Q_1)$ are bounded above by $d_S l(\partial R)$ 
where $d_S > 0$ is a positive constant. Moreover   any geodesic subarc of $\partial Q_1$ has at most two edges and no component of $\partial Q_1$ is geodesic. 
Recall that $\mathcal{P}$ is the   convex polygonal fundamental domain obtained as the complement of the system of curves $\lambda_0,\ldots,\lambda_{2g-1}$ on the surface $S$.

The boundary components of $Q_1$ are piecewise geodesic and map to the arcs $e_0,\ldots, e_{2g-1}$. Each geodesic boundary  arc of length one   is labeled by either $e_i, e'_i$ or their inverses. Disregarding edge orientations, each   arc of length two   is labeled by either $e_0e_0, e_{2g-1}e_{2g-1}, e_i e'_i,  e'_i e_i$. We say that a geodesic boundary arc has odd label or even label if $i$ is   odd or even, respectively. 

The proof proceeds in three steps.

\begin{enumerate}

\item \emph{Identifying length two geodesic boundary arcs with even labels.}
Consider a geodesic boundary arc $\alpha$ of length two and even label.	We will assume that with the induced orientation  $\alpha$   reads $e_{i}e'_{i}$  with $i$ even. The other cases are treated analogously. 
	
	The midpoint of the length two arc $\alpha$ has an incident edge contained in the interior of $Q_1$ and labeled $e_{i+1}$.  Trace a geodesic arc $\gamma$  on the surface $Q_1$ starting at this edge with labels alternating between $e_{i+1}$ and $e'_{i+1}$. It runs transverse to geodesic arcs of the two forms $e_{i}e'_{i}$ and $e_{i+2}e'_{i+2}$ until it reaches the boundary of $Q_1$ again, necessarily at the midpoint of some length two geodesic arc. 
	
	There are two possibilities to consider. If $\gamma$ reaches the boundary of $Q_1$ at the midpoint of a geodesic arc $\beta$ labeled $e_{i}e'_{i}$ then we may simply identify the two arcs $\alpha$ and $\beta$ to reduce the number of even boundary components of length two.
	
	\editA{	
	The other possibility is that $\gamma$ reaches the midpoint of a geodesic arc $\beta$ labeled $e_{i+2}e'_{i+2}$. Without loss of generality, let us assume that the last edge along $\gamma$ is labeled $e_{i+1}$. If this is the case, let $D$ be the $\mathcal{P}$-surface obtained by gluing two copies of the polygon $\mathcal{P}$ along the edge $e'_{i+1}$ so that $D$ has two boundary geodesic arc labeled $e_i e'_i$ and $e_{i+2} e'_{i+2}$. Attach the boundary labeled $e_{i+2} e'_{i+2}$ in $D$ to the arc $\beta$ in $Q_1$. The resulting $\mathcal{P}$-surface is locally convex by Proposition \ref{prop:identifying same edges}.
	The geodesic $\gamma$ can now be extended by one  edge, so that it   reaches the midpoint of the boundary geodesic arc labeled by $e_{i}e'_{i}$ in $D$. We can now proceed as before.} 	
	
	While this operation might increase the total length of geodesic boundary arcs with odd labels or the total number of edges with even labels, it will reduce the number of geodesic boundary arcs of length two and even label. 
	

	Let $Q_2$ denote the resulting surface. We point out that $Q_2$ is again a $\mathcal{P}$-surface with locally convex boundary by Proposition \ref{prop:identifying same edges}.

	\item \emph{Identifying the remaining geodesic boundary arcs of length one with even labels.}
	 Every geodesic boundary arc of $Q_2$ with an even label has length one. Moreover $Q_2$ has the same number of edges labeled $e_i$ and $\bar{e}_i$ for every even $i$ by Proposition \ref{prop:sides match up}. Choose any bijection between these two sets and identify edges in pairs. Let $Q_3$ denote the resulting surface. Once again Proposition \ref{prop:identifying same edges} implies that $Q_3$ is a $\mathcal{P}$-surface with locally convex boundary. 

\item \emph{Attaching a surface along the totally geodesic boundary curves with odd labels.}
Every boundary geodesic arc of $Q_3$ is locally convex and has an odd label. This implies that $Q_3$ has totally geodesic boundary and every boundary component maps to a power of one of the non-separating simple closed geodesic curves $\lambda_i$ with some orientation and $i$ odd. 

One direction of the Corollary \ref{cor:prescribed boundary} to Neumann's lemma implies that the number of boundary components of $Q_3$ is even and that the total degree over each  curve $\lambda_i$ with one orientation is equal to the total degree over $\lambda_i$ in the opposite orientation. This fact and the other direction of Corollary \ref{cor:prescribed boundary} allows us to construct a $\mathcal{P}$-surface $T$ with $\mathrm{Area}(T)$ linearly bounded in $l(Q_3)$ whose boundary components are the same as $Q_3$ but with opposite orientation. Let $Q$ be the cover of $S$ obtained by identifying the boundary of $Q_3$ with that of $T$ in the obvious way.

%
%
%
\end{enumerate}
To conclude observe that $R$ embeds in the cover $Q$ making Diagram (\ref{eq:commute R}) commute and that  $\mathrm{Area}(Q\setminus R)$ is bounded above linearly in   $l(\partial Q_1)$ and hence in $l(\partial R)$, as required.
\end{proof}

\section{Proof of geometric flexible stability}

Let $S$ be a compact hyperbolic surface. Consider some $*$-cover $p : X \to S$. We rely on Theorem \ref{thm:variant of LERF} established in   the previous section to complete the proof of Theorem \ref{thm:surface groups are stable} and therefore of our main result Theorem \ref{thm:main result}.

Our strategy is   to construct a cut graph $\Gamma$ on $X$ and then cap-off the complement $X \setminus \Gamma$ to obtain an unramified cover by making use of Theorem \ref{thm:variant of LERF}.
%
%

Since $X$ need not be a cover in the usual sense it might   contain   certain pathological closed curves that cannot exist on a cover of $S$. For example $X$ might admit  a  simple closed curve $\gamma$  such that a lift of $p \circ \gamma$ to $\widetilde{S}$ admits self-intersections. It is clear that if $X$ is   $\varepsilon$-reparable then such a curve has to be eliminated. This   motivates the following.

\begin{prop}
	\label{prop:cut surface embedds in a cover}
Let $\Gamma$ be a cut graph on $X$  and $C$ be any connected component of $X \setminus \Gamma$. Then $C$ embeds in some cover of $S$ as a subsurface with locally convex boundary.
\end{prop}



\begin{proof}
	Let $C$ be a connected component of $X \setminus \Gamma$. It follows from the definition of cut graphs that $C$ has no singular points of $s(X)$ in its interior and that its boundary is locally convex. In particular $C$ is non-positively curved and the map $f = p_{|C} :C \to S$ is a local isometry. The result  follows from Lemma \ref{lem:local isometry and locally convex boundary}.
\end{proof}

We are now ready to prove that closed hyperbolic surfaces are flexibly geometrically stable.

\begin{proof}[Proof of Theorem \ref{thm:surface groups are stable}]
Let the constant $\varepsilon > 0$ be given. Take $c > 0$ to be sufficiently small so that  $ c  \max \{1, b_S\}  \le \varepsilon$ where $b_S$ is the constant as in Theorem \ref{thm:variant of LERF}. Let $\delta = \delta(c)$ be the constant provided by Theorem \ref{thm:existance of cut graphs} such      that any $*$-cover $X$ with $\beta(X) < \delta \mathrm{Area}(X)$ admits a $c$-cut graph.

We claim that $\delta$ is as required in the definition of flexible geometric stability with respect to the given $\varepsilon > 0$. To see this consider  a $*$-cover $p : X \to S$ with $\beta(X) < \delta \mathrm{Area}(X)$. We need to show that $X$ is $\varepsilon$-reparable.
 Let $\Gamma$ be a $c$-cut graph on $X$.   In particular 
 $$
 l(\Gamma) \le c \mathrm{Area}(X) \le \varepsilon \mathrm{Area}(X).$$ 
  Consider the complement $C =X \setminus \Gamma$. Every connected component of $C$  isometrically embeds  into some cover of $S$ according to  Proposition \ref{prop:cut surface embedds in a cover}.  Relying on Theorem   \ref{thm:variant of LERF} we find a cover $q : Q \to S$ admitting an isometrically embedded copy of $C$ on which the restriction of   $q$ agrees with the map $p$, namely Diagram (\ref{eq:embeds diagram}) commutes. Moreover 
  $$ \mathrm{Area}(Q) - \mathrm{Area}(C) \le  b_S  l(\partial C) \le b_S c \mathrm{Area}(X) \le \varepsilon \mathrm{Area}(X).$$
This completes the verification that $X$ is indeed   $\varepsilon$-reparable.
\end{proof}

The fact that Theorem \ref{thm:surface groups are stable} implies our main result Theorem \ref{thm:main result} is contained in Proposition \ref{prop:geometric stability implies algebraic stability}.

%
\appendix
\section{Voronoi   and   Delaunay on singular planes}

\label{sec:properties of Voronoi and Delaunay}

We generalize some of the basic properties of the Voronoi tessellation and the Delaunay graph from the classical   Euclidean and hyperbolic cases to the framework of a "hyperbolic plane  with singularities". 

In particular we present the detailed proofs of Propositions \ref{prop:properties of Voronoi and Delaunay}, \ref{prop:two Voronoi cells are disjoint or share a common side}, \ref{prop:Delaunay graph is embedded} and \ref{prop:the angle between two edges of Delaunay graph is at most pi} that were  merely stated  without a proof in Section   \ref{sec:hyperbolic planes with singularity}. 

%

\subsection*{Hyperbolic planes with singularties}
Let $S$ be a compact hyperbolic surface and  $p : X \to S$   a fixed  $*$-cover. Let $q : \widetilde{X} \to X$ denote the universal cover of $X$ equipped with the pullback length metric $d_{\widetilde{X}}$. The singular set of $\widetilde{X}$ is  given by $s(\widetilde{X}) = q^{-1}(s(X))$. 
The \emph{Voronoi cell} $A_v$ at any vertex $v \in s(\widetilde{X})$ and the \emph{Delaunay graph} $D$ were defined in Section \ref{sec:hyperbolic planes with singularity}.

We find it useful to introduce the following additional notation. Let $m(v,u)$ denote the set of \emph{midpoints}
$$  m(v, u) = \{x \in \widetilde{X}\: : \: d_{\widetilde{X}}(x,v) = d_{\widetilde{X}}(x,u) \} $$ 
between any two given points $v,u \in \widetilde{X}$. It is well-known   \cite[\S7.21]{beardon2012geometry} that the set of midpoints of any pair of  points in the hyperbolic plane is the perpendicular bisector to the geodesic arc connecting these two points.

Two vertices $v_1$ and $v_2$ of the Delaunay graph are connected by an edge   if and only if there exists a mid-point $x \in m(v_1,v_2)$ so that $d(x,u) > d(x,v_1) = d(x,v_2)$ for every other vertex $u\in s(\widetilde{X}) \setminus \{v_1,v_2\}$.

\begin{lemma}
\label{lem:uniqueness of circumcenter}
Let $x_1,x_2,x_3 \in \widetilde{X}$ be three distinct   points.  Then there is at most a single closed metric ball $B \subset \widetilde{X}$ with $\mathring{B} \cap s(\widetilde{X}) = \emptyset$ and $ \{x_1, x_2, x_3\} \subset \partial B$.
\end{lemma}

\begin{proof} 
Assume that $B \subset \widetilde{X}$ is a closed ball as in the statement of the lemma. Let $T \subset \widetilde{X}$ be the geodesic triangle with vertices $x_1,x_2,x_3$. The convexity of the ball $B$ implies   that $T \subset B$. Making use of Corollary \ref{cor:convex subset embeds into H} we may regard $B$ as being isometrically embedded in the hyperbolic plane. The center $x_0$ of the ball $B$ is determined by the triangle $T$. More precisely $x_0$ is the mutual intersection point of the three mid-point bisectors to the edges of $T$. The same reasoning shows that  any other ball $B' \subset \widetilde{X}$ as in the statement of the lemma   necessarily has the same center as $B$   and must therefore     agree with $B$.
\end{proof}

\medskip

Proposition \ref{prop:properties of Voronoi and Delaunay} says that every Voronoi cell is homeomorphic to a disc, is convex with a piecewise geodesic boundary and is determined by its Delaunay neighbors.

\begin{proof}[Proof of Proposition \ref{prop:properties of Voronoi and Delaunay}]
Let $v \in s(\widetilde{X})$ be a vertex of the Delaunay graph. The distance function $d_{\widetilde{X}}(\cdot, u)$ is continuous for every point $u \in \widetilde{X}$. Therefore the Voronoi cell $A_v$ is closed. It is clear from the definition  that $A_v \cap s(\widetilde{X}) = \{v\}$. Since the singular set $s(\widetilde{X})$ is co-bounded  it follows that $A_v$ is   bounded and hence compact.  Therefore  there  is a finite subset of vertices $\{u_1, \ldots, u_n\} \subset s(\widetilde{X})$ for some $n \in \mathbb{N}$
such that the Voronoi cell $A_v$ is given by
$$  A_v = \{x \in \widetilde{X}\: : \: d_{\widetilde{X}}(x,v) \le d_{\widetilde{X}}(x,u_i)\quad \forall i \in \{1,\ldots,n\} \}.$$
If $x \in \partial A_v$ is a boundary point of $A_v$ then   at least one of the above inequalities must be an equality. In other words
$ \partial A_v \subset A_v \cap \bigcup_{i=1}^n m(v,u_i)$.

For every boundary point $x \in \partial A_v$ let $F_x \subset \{u_1,\ldots,u_n\}$ be the subset given by   $u_i \in F_x$ if and only if $x \in m(v,u_i)$. Note that $F_x $ is non-empty for every $x \in \partial A_v$.   Every boundary point  $x \in \partial A_v$ admits a closed ball $B_x \subset \widetilde{X}$ centered at $x$ with $\mathring{B}_x \cap s(\widetilde{X}) = \emptyset$ and $\partial  B_x \cap s(\widetilde{X}) = F_x \cup \{v\}$. This relies on the fact that $x \in A_v$ and therefore 
   $d(x,v) \le d(x,u)$ for \emph{every} vertex $u \in s(\widetilde{X})$. 
   
   Every boundary point $x \in \partial A_v$ admits  an open neighborhood $U_x$ so that 
$$ A_v \cap U_x =  \{y \in U_x \: : \: d_{\widetilde{X}}(y,v) \le d_{\widetilde{X}}(y,u_i)\quad \forall i \in F_x \}. $$
Assume without loss of generality  that $U_x$ is sufficiently small so that $U_x \subset B_x$. 
   The ball $B_x$ is convex  and can be regarded   as being isometrically embedded in the  hyperbolic plane  by Corollary \ref{cor:convex subset embeds into H}.  We conclude that $ A_v \cap U_x$ can be isometrically identified with a  neighborhood of the point $x$ inside a \emph{hyperbolic Voronoi cell}.
   
Note  that $|F_x| = 1$ for all but finitely many boundary points $x \in \partial A_v$ according to    Lemma \ref{lem:uniqueness of circumcenter}.  This condition is open in the sense that given $x \in \partial A_v$ with $|F_x| = 1$ we have $F_x = F_y $ for every $y \in A_v \cap U_x$. We conclude that  $\partial A_v$ is locally isometric to a geodesic segment away from   finitely many points where $|F_x| > 1$. Moreover $\partial A_v$ is locally convex at those points as well. In particular the boundary of $A_v$ is piecewise geodesic.

The above description of the boundary of  $  A_v$ in terms of the    local hyperbolic picture with respect to the open cover   $U_x$ shows that $A_v$ is locally convex. The Cartan--Hadamard theorem implies that $A_v$ is convex, see e.g. Lemma \ref{lem:lifting local isometry}. The exponential map at $v$ sets up an homeomorphism of $A_v$ to a disc. This concludes the proof of Items (\ref{it:Voronoi cell is homeomorphic to a disc}) and (\ref{it:Voronoi has pw geodesic boundary}).

Note that a vertex $u$ is adjacent to $v$ in the Delaunay graph if and only if $u = u_i$ for some $ i\in\{1,\ldots,n\}$ and there is a boundary point $ x \in \partial A_v$ with $F_x = \{u_i\}$. In other words, the Delaunay neighbors of $v$ correspond to the geodesic sides of the cell $A_v$, and Item (\ref{it:Voronoi defined in terms of neighbors}) of Proposition \ref{prop:properties of Voronoi and Delaunay} follows.
\end{proof}

Proposition \ref{prop:two Voronoi cells are disjoint or share a common side} says that the interiors of distinct Voronoi cells are disjoint.

\begin{proof}[Proof of Proposition \ref{prop:two Voronoi cells are disjoint or share a common side}]
Assume that the intersection   $A_{v,u} = A_v \cap A_u$ is non-empty. 
Every  point $x \in A_{v,u} $ admits a closed metric ball $B_x \subset \widetilde{X}$ centered at $x$ with $\mathring{B}_x \cap s(\widetilde{X}) = \emptyset$ and $ \{v,u\} \subset \partial B_x$. The ball $B_x$ is convex and  so  can be regarded as being isometrically embedded in  the hyperbolic plane by Corollary \ref{cor:convex subset embeds into H}. 

Hyperbolic geometry implies that $m(v,u) \cap B_x$ is a geodesic segment for every point $ x \in A_{v,u}$. Since $A_{v,u} \subset m(v,u)$   the convexity of   $A_{v,u} $ implies that the intersection $A_{v,u} \cap B_x$ is a geodesic segment as well. The compactness of $A_{v,u} $ allows us to extract a finite cover using   such metric balls. It follows that $A_{v,u} $ is a finite union of geodesic segments. Since $A_{v,u} $ is convex it must be  a single point or a   single geodesic segment. 
In particular $A_{v,u}$ has empty interior and    $\mathring{A}_v \cap \mathring{A}_u = \emptyset$. Since both $A_v$ and $A_u$ are topological discs this implies   $\partial A_v \cap \mathring{A}_u = \mathring{A}_v \cap \partial A_u = \emptyset$. We conclude that  $A_{v,u} \subset \partial A_v \cap \partial A_u$ and that $A_{v,u}$ is a common geodesic side of the Voronoi cells $A_v$ and $A_u$. 

The second statement of Proposition \ref{prop:two Voronoi cells are disjoint or share a common side} concerning adjacency in the Delaunay graph     follows from  the last paragraph of the proof of Proposition \ref{prop:properties of Voronoi and Delaunay}.
\end{proof}

Proposition  \ref{prop:Delaunay graph is embedded} deals with the Delaunay graph and shows it is embedded.
\begin{proof}[Proof of Proposition \ref{prop:Delaunay graph is embedded}]
Let $v_1,u_1$ and $v_2,u_2$ be vertices of the Delaunay graph so that $v_i$ is adjacent to $u_i$.   Let $x_1$ and $x_2$ be points in $\widetilde{X}$ such that for $i \in \{1,2\}$ there is a closed metric ball $B_i \subset \widetilde{X}$ centered at $x_i$ and with $\mathring{B}_i \cap s(\widetilde{X}) = \emptyset$ and $\partial B_i \cap s(\widetilde{X}) = \{v_i, u_i\}$. Let $\gamma_i$ with $i \in \{1,2\}$ be the geodesic arc realizing the Delaunay edge between $v_i$ and $u_i$ and connecting these two points  in $\widetilde{X}$. In particular $\gamma_i \subset B_i$. 

Assume towards contradiction that the two chords $\gamma_1$ and $\gamma_2$ intersect non-trivially along their interior. Recall that a pair of geodesics in a $\mathrm{CAT}(-1)$-space admit at most a single intersection point. An  examination of the resulting planar diagram shows that $|\partial B_1 \cap \partial B_2| > 2$. Therefore   Lemma \ref{lem:uniqueness of circumcenter} applied  with respect  to any three distinct points of the intersection $\partial B_1 \cap \partial B_2$ shows that the balls $B_1$ and $B_2$ must coincide. This is  a contradiction.


The fact that the projection $q(D)$ is embedded in $X$ follows immediately from the above discussion combined with  the $\pi_1(X)$-invariance of the Delaunay graph.
\end{proof}

Proposition  \ref{prop:the angle between two edges of Delaunay graph is at most pi} shows that the connected components of the complement of the Delaunay graph have locally convex boundary.

\begin{proof}[Proof of Proposition \ref{prop:the angle between two edges of Delaunay graph is at most pi}]
Assume towards contradiction that the boundary of some connected component of $\widetilde{X} \setminus D$ is not locally convex. Equivalently,  there is some Delaunay vertex $v \in s(\widetilde{X})$ and some half-space $\mathfrak{H}$ with $v \in \partial \mathfrak{H}$   such that no Delaunay edge incident at $v$ is contained in $\mathfrak{H}$. The proof of Lemma \ref{lem:hyperbolic geometry lemma}  and in particular the Claim contained in that proof shows that the Voronoi cell $A_v$ has infinite area\footnote{While Lemma \ref{lem:hyperbolic geometry lemma} appears below Proposition \ref{prop:the angle between two edges of Delaunay graph is at most pi} in the text, the proof of that lemma is independent of this proposition.}. This contradicts the fact that $A_v$ is compact.
\end{proof}

\bibliography{testability}
\bibliographystyle{alpha}

\end{document}